\RequirePackage{fix-cm}

\documentclass[12pt]{article}
\usepackage[left=1in,right=1in]{geometry}
\usepackage{graphicx}
\usepackage{lipsum}
\usepackage{amsthm}
\usepackage{amsfonts,amsmath}
\usepackage{graphicx}
\usepackage{epstopdf}
\usepackage{algorithmic}
\usepackage{enumitem}

\usepackage{hyperref}
\usepackage[capitalise,nameinlink]{cleveref}
\usepackage{thmtools}
\usepackage{thm-restate}  %

\usepackage[group-minimum-digits=3,digitsep=comma]{siunitx}

\usepackage{amsopn}
\newtheorem{theorem}{Theorem}
\newtheorem{lemma}[theorem]{Lemma}
\newtheorem{corollary}[theorem]{Corollary}
\newtheorem{remark}[theorem]{Remark}
\newtheorem{assumption}[theorem]{Assumption}
\newtheorem{example}[theorem]{Example}
\Crefname{assumption}{Assumption}{Assumptions}

\DeclareMathOperator{\diag}{diag}

\DeclareMathOperator{\supp}{supp}
\DeclareMathOperator{\prox}{prox}
\DeclareMathOperator*{\argmin}{argmin}

\newcommand{\abs}[1]{| #1 |}
\newcommand{\nor}[1]{\left\Vert #1 \right\Vert}
\newcommand{\norm}[2]{\left\Vert #1 \right\Vert_{ #2 }}
\DeclareMathOperator{\Id}{Id}

\newcommand{\X}{\mathcal{X}}
\newcommand{\Y}{\mathcal{Y}}
\newcommand{\PP}{\mathcal{P}}
\newcommand{\DD}{\mathcal{D}}
\newcommand{\CC}{\mathcal{C}}
\newcommand{\Ps}{\mathcal{P}^{\sol}}
\newcommand{\Ds}{\mathcal{D}^{\sol}}
\newcommand{\bbs}{{b}^\star}
\newcommand{\bb}{b^\delta}
\newcommand{\LL}{\mathcal{L}}
\newcommand{\sol}{\star}
\newcommand{\LLs}{\mathcal{L}^{\sol}}
\newcommand{\LLd}{\mathcal{L}^{\delta}}
\newcommand{\Ss}{\mathcal{S}}
\newcommand{\seq}[1]{\left(#1\right)}

\newcommand{\xs}{{x}^\star}
\newcommand{\ys}{{y}^\star}
\newcommand{\zs}{{z}^\star}

\newcommand{\T}{T}
\newcommand{\bbR}{\mathbb{R}}
\newcommand{\R}{\mathbb{R}}
\newcommand{\N}{\mathbb{N}}

\newcommand{\paragraphfont}[1]{\textit{{#1}}}
\usepackage{xcolor}
\definecolor{linkcolor}{RGB}{83,83,182}
\definecolor{citecolor}{RGB}{128,0,128}
\usepackage{hyperref}
\hypersetup{
    colorlinks=true,
    citecolor=citecolor,
    linkcolor=linkcolor,
    urlcolor=linkcolor
}

\newcommand\blfootnote[1]{%
  \begingroup
  \renewcommand\thefootnote{}\footnote{#1}%
  \addtocounter{footnote}{-1}%
  \endgroup
}

\begin{document}

\title{Iterative regularization for low complexity regularizers}

\author{Cesare Molinari\textsuperscript{1}, Mathurin Massias\textsuperscript{2}, Lorenzo~Rosasco\textsuperscript{2,3,4}, Silvia Villa\textsuperscript{1}
}

\date{
    \textsuperscript{1}MaLGa, DIMA, Universit\`a di Genova \\
    \textsuperscript{2}MaLGa, DIBRIS, Universit\`a di Genova \\
    \textsuperscript{3} Center for Brains, Minds and Machines, MIT  \\
    \textsuperscript{4}  Istituto Italiano di Tecnologia \\
}

\maketitle

\begin{abstract}
	Iterative regularization exploits the implicit bias of an optimization algorithm to regularize ill-posed problems.
    Constructing algorithms with such built-in regularization mechanisms is a classic challenge in inverse problems but also  in modern machine learning, where it provides both a new perspective on algorithms analysis, and significant speed-ups compared to explicit regularization.
	In this work, we propose and study the first iterative regularization procedure able to handle biases described by non smooth and non strongly convex functionals,  prominent in low-complexity regularization.
	Our approach is based on a primal-dual algorithm  of which we analyze convergence and stability properties, even in the case where the original problem is unfeasible. The general results are illustrated considering the special case
 of sparse recovery with the $\ell_1$ penalty.
	Our theoretical results are complemented by experiments showing the computational benefits of our approach.
\end{abstract}

\blfootnote{
This material is based upon work supported by the Center for Brains, Minds and Machines (CBMM), funded by NSF STC award CCF-1231216.
L. R. acknowledges the financial support of the European Research Council (grant SLING 819789), the AFOSR projects FA9550-18-1-7009, FA9550-17-1-0390 and BAA-AFRL-AFOSR-2016-0007 (European Office of Aerospace Research and Development), and the EU H2020-MSCA-RISE project NoMADS - DLV-777826.
C. M. e S. V. are members of the INDAM-GNAMPA research group.}
\section{Introduction}

Parameters of machine learning models are frequently estimated by minimizing the sum of a data fidelity term and a regularization term: the datafit ensures that the model learns from the training data while the regularizer enforces %
good generalization \cite{shai}.
In this \emph{explicit} regularization framework, the regularization strength is controlled by a scalar parameter balancing the two terms.
To tune it, the most popular approach is grid-search: a grid of values is chosen, for each of which a model is obtained by solving the corresponding optimization problem \cite[Chap. 7]{Hastie_Tibshirani_Friedman09}.
Amongst these models, the best is then selected as the one minimizing a given criterion, such as AIC  \cite{Akaike74}, BIC \cite{Schwarz78} or error on left-out data \cite{Devroye_Wagner1979}.
The drawback of this widely used procedure is its cost: it requires solving as many optimization problems as regularization parameters on the grid.

In the wake of the practical successes of deep learning, there has been a recent surge of interest for an alternative, namely \emph{iterative} regularization.
Contrary to explicit regularization, it consists in solving a single optimization problem: the regularization is built into an iterative algorithm, and the regularization strength is controlled by the number of iterations \cite{kaltenbacher2008iterative}. %
Since a single problem is solved, and the algorithm typically stopped before convergence, iterative regularization can provide great computational speed-ups compared to explicit regularization.
It is closely related to implicit regularization, which refers to the fact that an algorithm is biased towards certain solutions of the problem it solves \cite{chizat2020implicit,gunasekar2017implicit}.
As a seminal example, under-determined least squares have infinitely many solutions, yet gradient descent initialized at zero converges to the minimal Euclidean norm one \cite[Chap. 6]{engl1996regularization}.
In a potentially complex loss landscape, this guides the search amongst all solutions to a specific one, allowing iterative regularization to be developed for the squared norm regularizer  \cite{yao2007early,raskutti,paglia}.
A question arises: for other regularizers, how to find an algorithm with adequate bias and iterative regularization properties?
In the case of strongly convex regularizers, iterative regularization has been investigated in two lines of work: the first one is based on mirror descent \cite{gunasekar2018characterizing,vavskevivcius2020statistical}, which can be viewed as dual gradient descent \cite{matet2017don}.
The second one, arising from the imaging community, is called linearized Bregman iterations \cite{cai2009convergence,yin2010analysis}.

However, many regularizers of interest are not strongly convex.
This is the case of so-called \emph{low complexity} regularizers: following pioneering work on the $\ell_1$ norm \cite{Chen_Donoho_Saunders98,Tibshirani96}, regularizers such as the nuclear norm \cite{Fazel02}, group norms \cite{Obozinski_Taskar_Jordan10} or Total Variation \cite{Rudin_Osher_Fatemi92} have been extensively used to obtain models exhibiting some notion of sparsity \cite{iutzeler2020nonsmoothness}.
In the explicit regularization framework, they have had a tremendous impact on machine learning \cite{Hastie_Tibshirani_Wainwright15}.
To use them in iterative regularization, some approaches exist, but they either are tailored to the $\ell_1$-norm \cite{Vaskevicius_Kanade_Rebeschini19}, or require tuning additional parameters \cite{yin2010analysis,YinOshBur08}.
Devising a generic and practical iterative regularization procedure for convex regularizers is thus still an open problem.
In this work,
\begin{itemize}[topsep=4pt]
	\item  we propose the first implementable iterative regularization procedure applicable to non smooth, non strongly convex regularizers,
	\item in the presence of noise, we derive a stopping time varying as the inverse of the noise level, in accordance with known results for strongly convex regularizers, %
	\item we handle the use of approximate computations and preconditioning in the algorithm,
	\item we provide a deeper analysis when the regularizer is the $\ell_1$ norm, and we obtain model recovery results,
	\item we validate our approach numerically and provide an open source python package.
\end{itemize}
\vspace{3mm}

The structure of the paper is as follows: we first formalize in \Cref{sec:background} the problem at hand and detail the notions of explicit and iterative regularization.
In \Cref{sec:framework} we present the algorithm we use for iterative regularization and the setup under which we analyze it.
In \Cref{sec:inexact}, we state our main result: stability in the presence of noise and a stopping time for iterative regularization.
\Cref{sec:related_works} contains a detailed comparison of our results to existing approaches.
\Cref{sec:ell1} is devoted to deeper results in the case of sparse recovery with the $\ell_1$ norm.
In \Cref{sec:unfeas} we study some cases of iterative regularization where the solution to the problem does not exist.
Experiments in \Cref{sec:experiments} demonstrate the validity of the approach.

\vspace{3mm}
\paragraph{Notation}
Let $\X$ be a real Hilbert space. For $\varepsilon \geq 0$, the $\varepsilon$-subdifferential of the function $f$ at the point $x\in\X$ is the set $\partial_\varepsilon f(x) = \{u \in \X : \forall y \in \X, f(x) - f(y) \leq \langle u, x - y \rangle + \varepsilon \}$; for $\varepsilon=0$ we write $\partial f(x)$.
For a symmetric positive definite $T$, $\norm{x}{T}^2 := \langle T^{-1}x, x\rangle$.
The $T$-preconditioned proximal operator of $f$ at $x$ is $\prox^T_f(x) = \argmin_{x'\in\X} f(x') + \frac{1}{2} \nor{x - x'}^2_T.$
The set of proper, convex and closed functions on the space $\X$ is denoted by $\Gamma_0(\X)$.
For a convex function $R$, $x'\in\X$ and $\theta \in \partial R(x')$, the Bregman divergence induced by $R$ with subgradient $\theta$ is defined as $D_R^\theta (x, x') = R(x) - R(x') - \langle \theta, x - x' \rangle$.
When $R$ is differentiable, its subdifferential at $x'$ reduces to $\{\nabla R(x')\}$ and thus, for the Bregman divergence, we omit the $\theta$ superscript.
The pointwise multiplication between vectors, or row-wise multiplication between a vector and a matrix, is denoted $\odot$.

\section{Background on explicit and iterative regularization}\label{sec:background}

As one motivation for our setting, consider the problem of learning a mapping $f$ between observations $(a_i, b_i)_{i \in [n]} \in \mathbb{R}^p \times \mathbb{R}$ such that $f(a_i) = b_i$.
In the case where the hypothesis space consists of linear functions, this amounts to learning a vector $x$ such that $\langle a_i, x \rangle = b_i$ for all $i \in [n]$, which, introducing the design matrix $A = (a_1^\top, \ldots, a_n^\top)^\top \in \mathbb{R}^{n \times p}$, means to solve
\begin{equation}\label{linsyst}
    Ax = b \enspace .
\end{equation}
Such inverse problems are ubiquitous in machine learning, signal processing and image processing.
Recent successful analyses of deep learning also consider linear approximations of kind \cite{montanari}.
It is common that the solution to \cref{linsyst} is not unique, for example in the overparametrized setting when $p > n$, common in machine learning, statistics and signal processing.
In this situation, amongst all possible solutions, it is popular to favor a particular one, e.g. considering:
\begin{equation}\label{eq:exact}
	\min_{x \in \X} R(x)  \quad \text{s.t.} \quad Ax = b \enspace,
\end{equation}
where the regularizer $R$ (also called penalty, or bias) selects the solutions of interest.
In this work, we are interested in a special type of regularizers, low complexity ones, which force the solution $x$ to lie on a reduced subset of the space, for instance on a lower dimensional manifold.
Critically, these regularizers are neither smooth, nor strongly convex (see, for instance, \cite{vaiter2015low,bach2012structured,iutzeler2020nonsmoothness}).
In \Cref{ex:sparse,ex:low_rank,ex:tv}, we recall some well-known examples;
other notable examples include the $\ell_\infty$ norm \cite{elvira2020safe}, ordered $\ell_1$ penalties \cite{figueiredo2016ordered} or block-sparse penalties \cite{Obozinski_Taskar_Jordan10,Simon_Friedman_Hastie_Tibshirani12}.

\begin{example}[Sparse regression and classification]\label{ex:sparse}
	When $\X = \R^p$, choosing $R(\cdot) = \nor{\cdot}_1$ corresponds to finding the minimal $\ell_1$-norm solution to a linear system, and in this case \eqref{eq:exact} is known as Basis Pursuit \cite{Chen_Donoho_Saunders98}.
	Following the practical success of compressed sensing \cite{Candes_Romberg_Tao06,Donoho06}, 	$\ell_1$-based approaches have had a tremendous impact in imaging, signal processing and machine learning in the last decades (see \cite{Hastie_Tibshirani_Wainwright15} for a review).
	Problem \eqref{eq:exact} also encompasses classification with the following rewriting:
	if one searches for the minimal $f$-valued separator to a linearly separable dataset $(a_i, b_i)$, the problem is:
	\begin{equation}\label{eq:classif}
        \min_{x \in \R^p} f(x) \quad \text{s.t.} \quad (b \odot A) x \succeq 1_n \enspace.
    \end{equation}
    Introducing a slack variable $u = (b \odot A) x - 1_n$, \eqref{eq:classif} fits in the framework of problem \ref{eq:exact} using $\tilde x =
    	\begin{pmatrix} x, u \end{pmatrix},\  R(\tilde x) = f(x) + \iota_{\{\cdot \succeq 0\}}(u)$, $\tilde A = \begin{pmatrix} b \odot A,  & - \Id
    	\end{pmatrix}$ and $\tilde b = 1_n$.
\end{example}
\begin{example}[Low rank matrix completion]\label{ex:low_rank}
	In many practical applications, such as recommender systems, one seeks to recover a partially observed matrix $B$ based on the assumption that its rank is low \cite{Fazel02,Candes_Recht09}.
	A convex approach to this problem is:
	\begin{equation}\label{pb:low_rank}
		\min_{X \in \bbR^{p_1 \times p_2}} \nor{X}_* \quad \text{s.t.} \quad X_{ij} = B_{ij}  \quad   \forall (i, j) \in \mathcal{D} \enspace,
	\end{equation}
	where $\nor{\cdot}_*$ is the nuclear norm and $\mathcal{D} \subset [p_1] \times [p_2]$ is the set of observed entries of the matrix $B$.
	In that case, $A$ is a self adjoint linear operator from $\mathbb{R}^{p_1 \times p_2}$ to $\bbR^{p_1 \times p_2}$, such that $(A X)_{ij}$ has value $X_{ij}$ if $(i, j) \in \mathcal{D}$ and 0 otherwise;
	the constraints write $A X = A B$.
\end{example}

\begin{example}[Total Variation]\label{ex:tv}
	In imaging tasks such as deblurring and denoising, regularization via Total Variation allows to simultaneously  preserve edges while removing noise in flat regions \cite{Rudin_Osher_Fatemi92}.
	Given a blurring operator $A: \ \X \to \X$, the problem of Total Variation is:
	\begin{equation}\label{eq:tv}
	    \min_{X \in \R^{p_1 \times p_2}} \nor{\nabla X}_{2,1} \text{ s.t. } A X = B\enspace.
	\end{equation}
	The above problem can be re-written as:
	\begin{equation}\label{pb:tv_v2}
		\min_{\tilde X \in \R^{(p_1 + p_2) \times p_2}} \; \Omega(\tilde X) \quad \text{s.t.} \quad  \tilde A \tilde X = \tilde B \enspace,
	\end{equation}
	with $\tilde X =
	\begin{pmatrix} X \\ U \end{pmatrix},\  \Omega(\tilde X) = \Vert U \Vert_{2,1}, \ \tilde A = \begin{pmatrix} A  &0 \\ \nabla &- \Id \end{pmatrix}$ and $\tilde B = \begin{pmatrix} B \\ 0 \end{pmatrix}$.
	To avoid increasing the dimension of the problem, one can also consider directly problem \eqref{eq:tv} and compute the proximal operator of TV iteratively, in which case it is necessary to handle errors in the prox as we will in \Cref{eq:x_update_is_prox} \cite{villa2013accelerated}.
\end{example}

Solving problem \eqref{eq:exact} thus allows to restrict the search of a solution to \Cref{linsyst} to a specific simple subset of the ambient space.
In practice, however, it is frequent that the observations are corrupted by noise: the true observations are only available through a noisy version $\bb$.
To avoid fitting the noise in the data, one should no longer impose the constraint $Ax = \bb$ and the approach \eqref{eq:exact} must be modified.
Explicit regularization consists in relaxing the equality constraint into a penalization, and solving a composite optimization problem:
\begin{equation}\label{eq:tikhonov}
	\min_{x \in \X} \frac12 \big\Vert Ax - \bb\big\Vert^2 + \lambda R(x) \enspace,
\end{equation}
where the nonnegative scalar $\lambda$ controls the trade-off between fitting the data and regularizing the solution.
As mentioned in the introduction, selecting the correct value for $\lambda$ is computationally costly.

Alternatively, it is possible to exploit the implicit bias of an optimization algorithm.
As a classical example, it is well-known \cite[Chap 6]{engl1996regularization} that iterations of gradient descent on least-squares,
\begin{equation}
    x_{k+1} = x_k - \gamma A^* (Ax_k - b) \enspace,
\end{equation}
converge\footnote{If initialized at 0 and provided $\gamma < 2 / \nor{A}_{\text{op}}^2$.} to the solution of \eqref{eq:exact} with $R = \frac{1}{2}\nor{\cdot}^2$.
When applied to $b=b^\delta$, iterative regularization consists in stopping gradient descent iterates before convergence.
What controls the regularization strength in this case is the number of iterations performed \cite{kaltenbacher2008iterative}.
Typically, when the noise level is of order of magnitude $\delta$, one seeks an implicitly biased algorithm and a stopping time $k(\delta)$ such that the algorithm, applied to $b^\delta$, produces iterates $(x_k)$ satisfying:
\begin{equation}
    D(x_{k(\delta)}, \xs) \leq \mathcal{O}(\delta^\alpha) \enspace,
\end{equation}
where $D$ is some discrepancy measure and $\xs$ is a solution of \eqref{eq:exact} with exact data.
As discussed next, it is the contribution of this paper to provide an algorithm, a stopping time and guarantees for a generic non-smooth convex regularizer $R$.

\section{Algorithm and assumptions}
\label{sec:framework}
In this section we present the algorithm we study and the mathematical assumptions we consider.
\subsection{Algorithm}
Let $\X$ and $\Y$ be real Hilbert spaces, $A\colon \X\to\Y$ a linear and bounded operator,  and $b\in\Y$.
Generalizing the above discussion, we consider the  following minimization problem, %
\begin{equation}\label{eq:our_pb}
    \min_{x \in \X} R(x) + F(x) \quad \text{s.t.} \quad Ax = b \enspace.
\end{equation}
The functions $R$ and $F$ are both assumed to be convex, proper, and lower-semicontinuous. In addition,  $F$ is  differentiable.
Compared to \eqref{eq:exact}, the splitting between a nonsmooth and a smooth term  allows us to handle the smooth term $F$ using only its gradient.

Let $\bbs \in \Y$ denote the exact observation, typically unavailable, and $\bb \in \Y$ denote the accessible noisy data.
We study a worst-case situation; namely, for some $\delta\geq 0$, we assume that
\begin{equation}
	\| \bb-\bbs \| \leq \delta \enspace.
\end{equation}

The algorithm we consider for iterative regularization is a preconditioned and inexact version of a  three steps primal-dual procedure  \cite{chambolle2011first,Condat13,Vu13} applied to the noisy data $b^\delta$. %
Given initializations $y_{-1}, \ y_0\in\Y$ and $x_0\in\X$, consider
\begin{align}\label{eq:algo}
    \begin{cases}
	\tilde{y}_{k} = 2 y_k - y_{k - 1} \enspace, \\
	x_{k+1} = \prox^{T, \, \varepsilon_{k + 1}}_R(x_k - \T\nabla F(x_k) - \T A^*\tilde y_{k})  \enspace,\\
	y_{k + 1} = y_k + \Sigma \left(Ax_{k + 1} - b^{\delta}\right) \enspace.
    \end{cases}
\end{align}
The first step is an extrapolation on the dual variable; the second one is the update of the primal variable and involves the proximal-point operator of $R$ and the gradient of $F$;
finally, the third step is the update of the dual variable, which accumulates the residuals of the constraint $Ax=b^{\delta}$. The operators
 $T\colon \X\to\X$ and $\Sigma\colon\Y\to\Y$ are linear, positive and bounded and can be intepreted as  preconditioners, or, if proportional to the identity, as step-sizes.
The proximal-point operator of $R$ is allowed to be computed inexactly with error $\varepsilon_{k + 1} \geq 0$, recovering the exact case for $\varepsilon_{k+1} = 0$.
The notation $\prox^{T,\, \varepsilon_{k+1}}_R$ is intended in terms of $\varepsilon$-subdifferential, namely %
\begin{equation}\label{eq:x_update_is_prox}
\begin{split}
x_{k+1} &= \prox^{T,\, \varepsilon_{k+1}}_R(x_k - \T\nabla F(x_k) - \T A^*\tilde y_{k})    \\
&\Longleftrightarrow
    -T^{-1}\left(x_{k + 1}-x_k\right) - \nabla F(x_k)- A^*\tilde{y}_k \in \partial_{\varepsilon_{k+1}} R(x_{k+1})   \enspace.
\end{split}
\end{equation}
To interpret algorithm~\eqref{eq:algo} as an instance of the approach in \cite{Condat13,Vu13}, it is useful to cast the update of the dual variable $y$  as a proximal step:
\begin{equation}\label{eq:y_update_is_prox}
\begin{split}
 	y_{k + 1} & = \argmin_{y \in \Y} \left\{\langle b^{\delta} - Ax_{k + 1}, y\rangle + \frac{1}{2} \norm{y - y_k}{\Sigma}^2\right\}  \\
	& = \argmin_{y \in \Y} \left\{\langle b^{\delta},y\rangle+\frac{1}{2}\norm{y-\left[y_k+ \Sigma Ax_{k + 1}\right]}{\Sigma}^2\right\} \\
	& = \prox^{\Sigma}_{\langle b^{\delta}, \cdot \rangle} \left(y_k+\Sigma Ax_{k+1}\right) \enspace.
\end{split}
\end{equation}
The above algorithm is cheap in terms of computations per iteration.
Indeed, it only requires one (inexact) evaluation of the proximal operator of the non-smooth function $R$, one evaluation of the gradient of the smooth function $F$ and one matrix-vector multiplication for $A$ and $A^*$.
Its memory cost is also minimal, as only one primal and two dual variables need to be stored.
The $\prox$ of $R$ can be computed exactly for many penalties of interest (see \cite{combettes2011proximal,mosci2010solving}). Through $\varepsilon_{k}$, our framework also handles the case where the optimization problem defined by the proximal operator is numerically computed, in an approximate fashion, through an iterative inner-routine (see \cite{bach2011optimization,salzo2012inexact,barre2020principled}).
\subsection{Assumptions}\label{sub:assumptions}
We first make the following general assumptions on functions and operators involved in the problem.
\begin{assumption}[General hypothesis]\label{ass:gen}
	$\X$ and $\Y$ are Hilbert spaces and $A: \X\to\Y$ is linear and bounded.
	The functions $R$ and $F$ belong to $\Gamma_0(\X)$, meaning that they are proper, convex and lower-semicontinuous.
	Additionally, $F$ is Fr\'echet-differentiable with $L$-Lipschitz continuous gradient on $\X$.
\end{assumption}
In order to introduce the next assumptions on the problem and the existence of an exact solution, %
we first define the set of primal solutions, the set of dual solutions, and the Lagrangian functional with respect to the exact datum $b^\star$:
\begin{align}
	&\Ps :=\argmin_{x \in \X} \left\{ R(x) + F(x): \ \ Ax=\bbs \right\} \enspace, \label{primal}\\
	&\Ds := \argmin_{y \in \Y} \left\{\left[R + F\right]^*(-A^*y) + \langle \bbs, y \rangle\right\}  \enspace,\label{dual}\\
	&\LLs(x, y) := R(x) + F(x) + \langle y, Ax - \bbs \rangle \enspace.
\end{align}
We also denote by $\mathcal{S}^\star$ the set of saddle-points of $\LLs$; namely, $(\bar{x},\bar{y})\in\mathcal{S}^\star$ if and only if $\LLs(\bar{x}, y) - \LLs(x, \bar{y}) \leq 0$ for every $(x, y) \in \X \times \Y$.
We write $\PP^\delta$, $\DD^\delta$, $\LL^\delta$ and $\mathcal{S}^{\delta}$ for their respective counterparts when $\bbs$ is replaced by $b^\delta$.
We refer to the corresponding problems and quantities as the exact and noisy ones, respectively. In the rest of the paper we will assume that one only has access to the noisy quantities: hence our focus  is on the iterative algorithm designed to solve the noisy problem $\PP^\delta$, having in mind that the problem of interest is the exact one $\PP^{\star}$. We make the following assumptions on the existence of solution \emph{to the exact problem}. Notice that, on the other hand, we do not require the existence of solutions (or even feasibility) for the noisy one.
\begin{assumption}[Existence of exact solution]\label{ass:exist}
	There exists a saddle -\\
	point for the Lagrangian $\LLs$ ($\mathcal{S}^\star\neq \emptyset$); namely, a pair $(\xs,\ys)\in\X \times \Y$ such that, for every $(x, y) \in \X \times \Y$,
	$$\LLs(\xs, y) - \LLs(x, \ys) \leq 0 \enspace.$$
\end{assumption}
\begin{remark}\label{remark:cond}
	Under \Cref{ass:gen}, the following statements are equivalent:
	\begin{itemize}
	\item[$\bullet$] $(\xs, \ys)\in\mathcal{S}^{\star}$; namely, it is a saddle-point for the Lagrangian $\LLs$;
		\item[$\bullet$] $\xs$  and $\ys$ satisfy the following optimality conditions:
	\begin{equation}\label{pd_optcond}
		\begin{cases}
			-A^*\ys-\nabla F(\xs)\in \partial R(\xs) \enspace,\\
			A\xs=\bbs \enspace.
		\end{cases}
	\end{equation}
	\end{itemize}
	Moreover, either one of these properties implies that $\xs$ is a primal solution and $\ys$ is a dual solution; namely, $\mathcal{S}^\star\subseteq \Ps \times \Ds$.
	Under usual qualification conditions  \cite[Thm. 26.2]{Bauschke_Combettes11}, the converse is also true: if $\xs \in \Ps$ and $\ys \in \Ds$, then $(\xs, \ys)$ is a saddle-point; namely, $\mathcal{S}^\star = \Ps \times \Ds$.
\end{remark}

As far as the parameters of the algorithm are concerned, we make the following assumptions on the preconditioners $T$ and $\Sigma$.
\begin{assumption}\label{ass:TSigm}
	The operator $T: \ \X \to \X$ is linear, bounded, self-adjoint and positive with spectrum lower and upper bounded by $\tau_m > 0$ and $\tau_M$ respectively.
	The same holds for $\Sigma: \ \Y \to \Y$ with lower and upper bounds $\sigma_m > 0$ and $\sigma_M$.
\end{assumption}
\begin{assumption}\label{ass:omega}
    Define the quantity $\omega := 1 - \tau_M(L + \sigma_M \nor{A}^2)$. The parameters $\tau_M$ and $\sigma_M$ are chosen so that $\omega \geq 0$.
\end{assumption}
\begin{assumption}\label{ass:thetarho}
    For $0<\xi<1$ and $\eta > 1$, let $\theta := \xi - \tau_M(\xi L+ \sigma_M \nor{A}^2)$ and $\rho := \sigma_m(\eta - 1) - \sigma_M \xi \eta$. The parameters $\tau_M, \sigma_m, \sigma_M$ and the constants $\xi, \eta$ are chosen so that $\theta\geq 0$ and $\rho>0$.
\end{assumption}

Notice that \Cref{ass:thetarho} is stronger than \Cref{ass:omega}. We consider them separately because some of our results hold only for \Cref{ass:thetarho}, while for other it is sufficient \Cref{ass:omega}.
Anyway, for every value of $L\geq 0$ and $\norm{A}{}$, it is always possible to choose the algorithm parameters $\tau_M, \sigma_m, \sigma_M, \xi$ and $\eta$ so that \Cref{ass:thetarho} (and so \Cref{ass:omega}) is fulfilled.
Choosing $\xi=1/4$ and $\eta=3/2$, for instance, amounts to require $\sigma_M<(4/3)\sigma_m$ and $\tau_M\leq (L+4\sigma_M\norm{A}{}^2)^{-1}$. For simplicity, the two preconditioners can be taken diagonal or as $T=\tau \Id$ and $\Sigma=\sigma \Id$, where $\Id$ is the identity operator while $\tau$ and $\sigma$ are positive parameters representing the primal and dual stepsizes of the algorithm.
In this case $\tau_m=\tau_M=\tau$, $\sigma_m=\sigma_M=\sigma$ and \Cref{ass:thetarho} naturally simplifies to $\tau\leq \xi(\xi L+ \sigma \nor{A}^2)^{-1}$ for some $0<\xi<1$.
For example, if $F= 0$ and thus $L = 0$, one recovers the classical step-size condition for the algorithm of \cite{Chambolle_Pock11}, that is $\sigma \tau \nor{A}^2 < 1$.

In the framework introduced above, we now show that Algorithm \eqref{eq:algo} is well-suited to iterative regularization, by studying its convergence and stability properties. %
\section{Convergence, stability and early-stopping bounds}\label{sec:inexact}

In this section, we present the main results of the paper. First, we start with a generalization of a well-known result about convergence of primal-dual algorithms. We include it since it
highlights the implicit bias of our algorithm in the case of exact data and exact computations ($b^{\delta} = \bbs$ and $\varepsilon_{k}=0$). Indeed, we prove convergence to a solution of problem $\mathcal{P}^\star$, namely,
amongst all solutions to $Ax = \bbs$, Algorithm \eqref{eq:algo} converges to one with minimal regularizer value.

\begin{restatable}{proposition}{condatconvergence}\label{prop:convergence_cp}
	Assume that  \Cref{ass:gen,ass:exist} hold.
	Let $(x_k, y_k)$ be the sequence generated by iterations \eqref{eq:algo} applied to $b^{\delta} = \bbs$ under \Cref{ass:TSigm,ass:omega}.
	Let also $\varepsilon_{k}=0$ for every $k\in\N$.
	Then $(x_k, y_k)$ weakly converges to a pair in $\mathcal{S}^\star$.
	In particular, $(x_k)$ weakly converges to a point in $\Ps$.
\end{restatable}

\Cref{prop:convergence_cp} is a first step towards an iterative regularization procedure: it shows that in the absence of noise, iterations \eqref{eq:algo} converge to a solution of interest.
The proof, in \Cref{app:condat_pf}, is a generalization of the results in \cite{Condat13} to our case.
The case with preconditioning, but $F=0$, is treated in \cite{pock2011diagonal}; while the case of $F\neq 0$ but without the preconditioning can be found in \cite{Condat13,Vu13}.

The next step is to show that when only $\bb$ is available, one can approximate the exact solution by early stopping the iterations \eqref{eq:algo} with noisy data.
To this end, we prove stability results in terms of Lagrangian gap and feasibility, that allow to derive a stopping time depending on the noise level $\delta$.
Before stating our main result (\Cref{prop:early_stop}), we first highlight why the Lagrangian gap and the feasibility are adequate quantities to measure convergence of the primal variable. In the next lemma we show that, if they are both zero, the primal variable is a solution of $\mathcal{P}^\star$.

\begin{restatable}{proposition}{feaslagrangian}\label{prop:feas_lagrangian}
	Let $(\xs, \ys) \in \mathcal{S}^\star$  and $(x, y) \in \X \times \Y$ such that
	$\LLs(x, \ys) - \LLs(\xs, y) = 0$ and $A x = \bbs$.
	Then $(x, \ys) \in \mathcal{S}^\star$.
\end{restatable}
We call the quantity $\LLs(x, \ys) - \LLs(\xs, y)$ \emph{Lagrangian gap}, as it is always non negative since $(\xs, \ys)$ is a saddle point.
More specifically, it is equal to the Bregamn divergence $D^{-A^* \ys}_{R+F}(x, \xs)$, as we detail in the proof (\Cref{app:pf_feaslagrangian}). The latter has been often used as an optimality measure in this context, see e.s. \cite{burger2007error}.
However, we emphasize that, contrarily to the $\alpha$-strongly convex case (where $D_{R+F}^{-A^*y^*}(x, x^\star) \geq \frac{\alpha}{2} \nor{x-x^\star}^2$), a vanishing Lagrangian gap is not enough for the primal variable to be a solution of the primal problem.
For example, for $R(\cdot)= \nor{\cdot}_1$ and $F(\cdot)=0$, the quantity $\LLs(x, \ys) - \LLs(\xs, y)$ vanishes whenever $x$ and $\xs$ have the same support and sign (or simply when $x=0$), while the primal variable $x$ can still be arbitrarily far away from $\xs$ (see \Cref{fig:bregman_void}).
\begin{remark}[Comparison with duality gap] Another quantity that is often considered as optimality measure for primal-dual algorithms is
\[
\sup_{v\in B_2} \LLs(x, v) -\inf_{u\in B_1} \LLs(u, y),
\]
where $B_1\subseteq\X$ and $B_2\subseteq\Y$ are two bounded sets containing a primal-dual solution,  see for example \cite{Chambolle_Pock11}. We note that such a bound can be easily derived from our convergence bounds in the exact setting. As discussed in \cite{Chambolle_Pock11}, the choice $B_1$ and $B_2$ is tricky while our bound is more easily readable in our linearly constrained setting.
\end{remark}

\begin{figure}[t]
    \centering
    \includegraphics[width=0.7\linewidth]{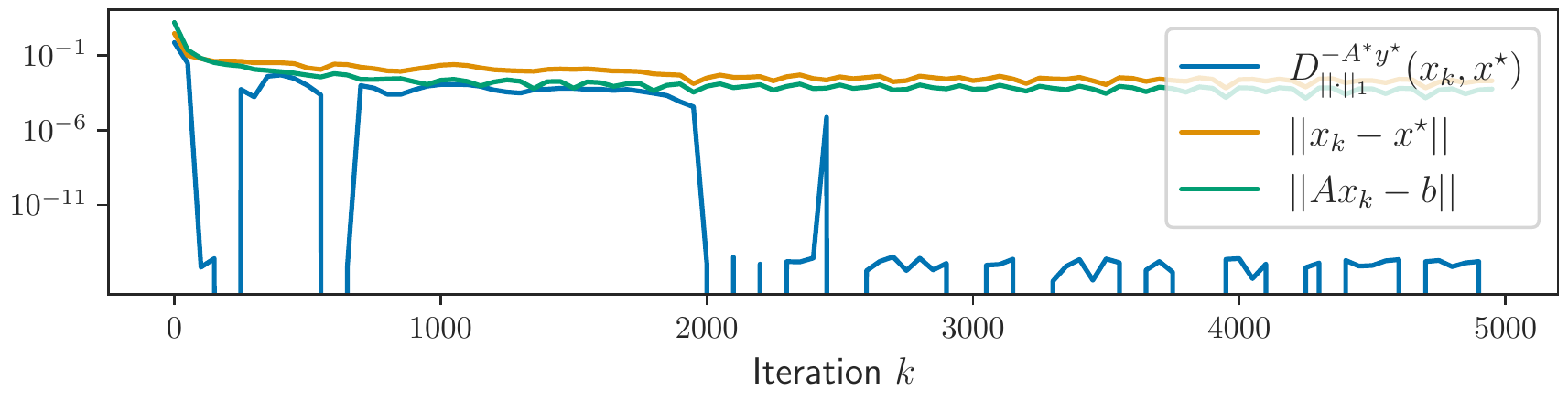}
    \caption{When $R + F$ is not strongly convex, the Bregman divergence alone is not enough to provide useful convergence rates.
    Here for $R = \Vert \cdot \Vert_1$, $F = 0$,  $D^{-A^*\ys}_R (x_k, \xs)$ vanishes quickly, while the iterates $x_k$ of \eqref{eq:algo} are still far from their limit $\xs$ (synthetic noiseless data, exact prox).}
    \label{fig:bregman_void}
\end{figure}

In the next result, we prove a stability bound for the iterates applied to the noisy problem, in terms of the optimality metric discussed above.

\begin{restatable}{theorem}{earlystop}\label{prop:early_stop}
	Let \Cref{ass:gen,ass:exist} hold and $(\xs, \ys) \in \mathcal{S}^\star$ be a saddle-point of the exact problem.
    Let $(x_k, y_k)$ be generated by \eqref{eq:algo} under \Cref{ass:TSigm,ass:omega} with inexact data $\bb$ such that $\nor{\bb - \bbs} \leq \delta$ and order-$\delta$ bounded error in the proximal operator, that is $|\varepsilon_k | \leq C_0 \delta$ for all $k \in \mathbb{N}$.
    Denote by $(\hat{x}_k, \hat{y}_k)$ the averaged iterates  $(\tfrac{1}{k}\sum_{j=1}^{k}x_j, \tfrac{1}{k} \sum_{j=1}^{k} y_j)$.
    Then there exist constants $C_1, C_2$, $C_3$ and $C_4$ such that, for every $k \in \N$,
	\begin{equation}\label{eq:lagrangian_bound}
		\begin{split}
			\LLs(\hat{x}_k,\ys) - \LLs(\xs, \hat{y}_k) \  \leq \ \frac{C_1}{k} +C_2 \delta + C_3 \delta^{3/2} k^{1/2} + C_4 \delta^2  k \enspace.
		\end{split}
	\end{equation}
	Let also \Cref{ass:thetarho} hold.
	Then there exist constants $C_5, C_6$, $C_7$, $C_8$ and $C_9$  such that, for every $k\in\N$,
	\begin{equation}\label{eq:feasibility_bound}
		\begin{split}
			\nor{A \hat{x}_{k} - \bbs}^2 \
			& \leq \frac{C_5}{k} + C_6 \delta+ C_7 \delta^{3/2} k^{1/2} +  C_8 \delta^2 k  + C_9 \delta^2  \enspace.
		\end{split}
	\end{equation}
\end{restatable}
The proof is given in \Cref{app:pf_earlystop}, where the reader can find also the explicit expression for all the constants involved in the bounds.
Note that the bounds \eqref{eq:lagrangian_bound} and \eqref{eq:feasibility_bound} are composed of two kinds of terms.
The first kind, related to \emph{optimization}, is of the form $\mathcal{O}(1/k)$ and vanishes with the iteration counter, as it is related to the convergence of the algorithm to the exact solution.
The second kind, involving $\delta$, is related to \emph{stability} and is due to the unavailability of $\bbs$.
In particular, when $\delta > 0$, the terms in $k$ make the bound increase with the iteration counter.

The main consequence of \Cref{prop:early_stop} is an early stopping procedure that allows to obtain upper-bounds on both Lagrangian gap and feasibility. %

\begin{corollary} Under the assumptions of \Cref{prop:early_stop},  setting
 $k = \tilde{C} / \delta$ for some constant $\tilde{C} > 0$, there exist constants $C$ and $C'$ such that
\begin{equation*}
	\begin{split}
	& \LLs(\hat{x}_{k},\ys) - \LLs(\xs, \hat{y}_{k}) \  \leq \ C \delta \enspace,\\
	& \nor{A \hat{x}_{k} - \bbs}^2 \ \leq C' \delta + C_6 \delta^2 \enspace.
	\end{split}
\end{equation*}
\end{corollary}

This result, combined with \Cref{prop:feas_lagrangian}, shows that the exact solution can be approximated by the averaged iterates generated by algorithm \eqref{eq:algo} on the noisy data, even if the true data is unavailable, by stopping at an appropriate iteration.
Assuming $\delta\leq 1$, the level of approximation between the early-stopped iterate and the exact solution is then proportional to the noise level $\delta$, both for the Lagrangian gap and the feasibility.
We provide further comments and comparisons with existing results  in the next section. We add one remark first.

\begin{remark}[Early stopping in absence of noisy solution]\label{sub:divergence}
We have shown that Algorithm \eqref{eq:algo}, with appropriate early-stopping strategies, provides a good approximation of the exact solution, even if the noiseless datum is unavailable.
Ill-posedness of the problem may be due to instability or non existence of the noisy solution. Our bounds in \Cref{prop:early_stop}  apply to both these situations.
If the problem is ill-posed from the stability point of view, the noiseless and noisy solutions are far apart, and the bounds in \Cref{prop:early_stop}
imply that early-stopping ensures a computationally efficient way to find a stable solution. If the noisy problem does not have a solution, the averaged primal iterates generated by Algorithm \eqref{eq:algo} may diverge (see the example in \Cref{app:divergence}). In this case, early-stopping is thus not only efficient to get a solution stable to noise, but indeed necessary to prevent unbounded behaviours. %
In this situation, it is thus mandatory to perform early-stopping, and this confirms that it is unavoidable to have a stability bound going to $+\infty$ with the number of iterations. For more results related to the unfeasible case,  see also \Cref{sec:unfeas}.
\end{remark}

\section{Comparison with existing results}
\label{sec:related_works}

The idea of exploiting the implicit regularization properties of optimization algorithms has been studied, often under the name of iterative regularization, in the fields of inverse problems \cite{engl1996regularization}, image restoration \cite{burger2007error}, and more recently machine learning \cite{yao2007early}.
Existing methods can be divided into two classes, depending on whether or not strong convexity of the regularizer is assumed.
In the following we compare known results with ours.

\subsection{Strongly convex regularizer}
\label{sub:strongly_convex}
We begin noting that, to the best of our knowledge, our method is the only one to handle the smooth term $F$ in the regularizer
using only its gradient. We next provide an overview of the algorithms proposed for iterative regularization.
\begin{itemize}
    \item
    \paragraphfont{ Gradient descent, stochastic or accelerated.}
    The study of implicit regularization properties of gradient descent, known in the inverse problem community as Landweber method, goes back to the 50's \cite[Chap. 6]{engl1996regularization}.
    Accelerated versions of gradient descent, first proposed by Nesterov in \cite{nesterov1983method}, have been also studied in inverse problems  \cite{neubauer2017nesterov}. Approaches related to the heavy-ball method \cite{polyak1964some} have also been considered in inverse problems under the name of $\nu$-method, see \cite{engl1996regularization}.
    Generalizations towards $p$ norms with $p>1$ have been considered \cite{schopfer2006nonlinear,brianzi2013preconditioned}, while more general choices are not as studied. Interestingly, there is a rich literature in the non-convex setting for nonlinear inverse problems \cite{kaltenbacher2008iterative}.
    These ideas have been extended to machine learning considering regularizing properties of gradient descent \cite{yao2007early}, and its stochastic and accelerated versions  \cite{moulines2011non,rosasco2015learning,paglia}.
    \item
    \paragraphfont{ Linearized Bregman iterations and mirror descent.}
    Interest in regularizers beyond the Euclidean norm, in particular non strongly convex ones, has been mainly motivated by imaging applications and Total Variation regularization.
    Following the pioneering work of \cite{OshBurGol05,lorenz2014linearized}, a series of methods have been designed for iterative regularization with general convex regularizers (see \cite{burger2007error} and references therein).
    If $R$ is $\alpha$-strongly convex, the iterative algorithm to exploit is mirror descent \cite{Nemirovski_Yudin83,BecTeb03}, which has been popularized in the inverse/imaging problems community under the name of ``Linearized Bregman iterations'' \cite{yin2010analysis,YinOshBur08}:
    \begin{equation}\label{eq:lin_breg_it}
        \begin{cases}
        x_{k+1} = \argmin_{x \in \X} D_{R - \frac{\alpha}{2} \Vert \cdot \Vert^2}^{p_k}(x, x_k) + \langle x, A^* (A x_k - b) \rangle + \frac{1}{2\alpha} \norm{x - x_k}{}^2 \enspace ,\\
        p_{k+1} = p_k - \frac{1}{\alpha}(x_{k+1} - x_k) - A^*(Ax_k - b)  \enspace .
        \end{cases}
    \end{equation}
    It has been shown that this algorithm, in combination with a discrepancy type stopping rule, regularizes ill-posed problems \cite{cai2009convergence}.
    \item
    \paragraphfont{Accelerated dual gradient descent.} From a different perspective, the stability and regularization properties of the accelerated variant of Linearized Bregman iterations have been studied in \cite{matet2017don}.
    In the latter, mirror descent is interpreted as gradient descent applied to the dual; this connection, without acceleration, can also be found in  \cite{yin2010analysis}.
    \item
    \paragraphfont{ Diagonal approaches.}
    All the aforementioned techniques are tailored to the use of a quadratic datafitting term.
    They cannot be applied when the nature of the noise differs, calling for another loss.
    In that case, diagonal approaches offer an alternative, applying an optimization algorithm to successive approximations of the original problem \cite{BahLem94}.
    Convergence rates and stability of diagonal approaches for inverse problems have been considered in \cite{garrigos2018iterative} and in \cite{calatroni2019accelerated} for the accelerated case.
\end{itemize}

\subsection{Non strongly convex regularizers}
If the regularizer is only convex, as we consider, Linearized Bregman iterations cannot be applied and one must resort to one of the following.

\begin{itemize}
    \item
    \paragraphfont{Bregman iteration and ADMM.}
    The main algorithm in this case is ADMM \cite{Boyd_Parikh_Chu_Peleato_Eckstein11}, which has been studied in the imaging community under the name of Bregman iterations.
    Starting from $x_0 = 0$ and $p_0 = 0$, its updates read
    \begin{equation}\label{eq:breg_it}
    \begin{cases}
        x_{k+1} \in \argmin_{x \in \X} D^{p_k}_R(x, x_k) + \frac{1}{2} \norm{Ax - b}{}^2 \enspace, \\
        p_{k+1} = p_k - A^*(Ax_{k+1} - b)  \enspace.
    \end{cases}
    \end{equation}
    The algorithm converges to the solution of \eqref{eq:exact}; its regularization properties can be found in \cite{burger2007error}.
    It has been extended to nonlinear inverse problems in \cite{bachmayr2009iterative}.
    However, this method is impractical, since the minimization step in $x$ cannot be performed exactly.
    \item
    \paragraphfont{Bregmanized Operator Splitting and linearized/preconditioned ADMM.}\\
    These variants of Bregman iterations and ADMM rely on preconditioning to avoid the resolution of a difficult optimization problem at each iteration.
    They have been used empirically as regularizing procedures in inverse and imaging problems \cite{ZhaBurOsh11,ZhaBurBre10}.
    While convergence results are known, we are not aware of any theoretical quantitative stability result.
    \item
    \paragraphfont{Specific algorithms for sparse recovery and compressed sensing.}
        In the specific context of sparse recovery, \cite{osher2016sparse} and \cite{Vaskevicius_Kanade_Rebeschini19} have devised specific optimization procedures.
        These do not generalize to regularizers beyond $\ell_1$, nor do they allow to handle $F$.
    \item
    \paragraphfont{Exact regularization ($F=0$).}
    Exact regularization \cite{friedlander2008exact,yin2010analysis,schopfer2012exact} refers to solving
    \begin{equation}
        \min_x \ R(x) + \frac{\alpha}{2} \norm{x}{}^2  \quad \text{s.t.} \quad Ax = b \enspace,
    \end{equation}
    and to showing that there exists a value of $\alpha$ such that this new problem and \eqref{eq:our_pb} have the same minimizer.
    Then, known iterative regularization algorithms for the strongly convex case can be applied.
    The main drawback is that the existence of such a value of $\alpha$ is not guaranteed in general, it is problem specific, and cannot be determined in advance; hence it becomes a value to be tuned, which in turn is costly. When the regularizer is given by the $\ell^1$-norm, this approach is also related to the one of sparse Kaczmarz method proposed in \cite{schopfer2019linear}.
\end{itemize}
\ \\
As clear from the above discussion, to the best of our knowledge, \emph{there previously did not exist an implementable iterative regularization procedure able to handle any non strongly convex regularizer}.
Our proposed method fills this gap, and can be applied to the many instances of non smooth non strongly convex regularizers.

\section{The special case of sparse recovery with $\ell_1$-norm}
\label{sec:ell1}

In this section, we strengthen the results of \Cref{sec:inexact} in the case of sparse recovery. The choice $R(\cdot) = \nor{\cdot}_1$ has had a tremendous impact on sparse model estimation \cite{foucart2013}.
Below, we specialize our results to this case, obtaining bounds not only in terms of Lagrangian gap and feasibility, but directly on the distance between the iterates and the true model. The main result of the section, \Cref{thm:model_recovery} is a corollary of our results and a lemma in \cite{grasmair2011necessary} which allows to control the distance between a point and a solution in terms of the feasibility and the langrangian gap.
Therefore, in the next subsection, we first recall some results in \cite{grasmair2011necessary}, while the new result is in \Cref{ell1itreg}.
\subsection{Sparse recovery and compressed sensing}
We set $\X = \ell^2(\N; \R)$, $R(\cdot) = \nor{\cdot}_1$ and $F(\cdot) = 0$.
The support of $x \in \X$ is $\supp(x):= \{i \in \N: \ x_i\neq 0\}$ and $\abs{\cdot}$ denotes the cardinality of a set.
The Bregman divergence induced by $\nor{\cdot}_1$ is simply denoted by $D$. We first recall some notions from \cite{grasmair2011necessary}.

\begin{restatable}{proposition}{extendedsupp}\label{prop:extended_supp}
Fix a primal-dual solution $\left(\xs, \ys \right) \in \mathcal{S}^\star$.
Let the extended support be $\Gamma  := \{i \in \N: \abs{\left(A^*\ys\right)_i} = 1 \}$
and the saturation gap be $m  := \sup \left\{\abs{\left(A^*\ys\right)_i}:  \abs{\left(A^*\ys\right)_i} < 1 \right\}$.
Then $\Gamma$ is finite, and $m < 1$.
Moreover, for every $x\in\X$, with $\Gamma_C := \N \setminus \Gamma$,
\begin{equation}\label{eq:bregman_ext_supp}
	\begin{split}
		D^{-A^*\ys}(x,\xs)
		&\geq (1-m) \sum_{i\in\Gamma_C}^{} \abs{x_i}.
	\end{split}
\end{equation}
\end{restatable}
\ \\
For completeness, the proof is reported in \Cref{app:sub:sparse_proof}.
As  $D^{-A^*\ys}(\xs, \xs)=0$ and $m<1$, the (finite) set $\Gamma$ can be considered as an \emph{extended support}, as \Cref{eq:bregman_ext_supp} shows that $\xs$ is zero on the indices of $\Gamma_C$.

More generally, if for some $x\in\X$ we have $D^{-A^*\ys}(x,\xs)=0$, then $x=0$ (and so $x$ coincides with $\xs$) on $\Gamma_C$.
On the other hand, as mentioned above, $D^{-A^*\ys}(x,\xs)=0$ does not ensure any similarity between the two vectors on $\Gamma$, the finite subset of indices where the components of $\xs$ may be non-zero (see \Cref{fig:bregman_void}).

\ \\
To obtain sparse recovery results, we rely on compressed sensing assumptions on the design operator $A$ and on the exact primal solution $\xs$.
Based on \Cref{ass:CS}, \Cref{lem:CS} will allow us to bound $\nor{x - \xs}$ by a combination of the feasibility and the Lagrangian gap.
\begin{assumption}[Compressed sensing]\label{ass:CS}
    \ For some $s\in\mathbb{N}$,
    \begin{itemize}
    	\item[1. ] there exists a $s$-sparse solution $\xs$ to \cref{linsyst}; namely, $A\xs=\bbs$ with $ |\supp(\xs) | \leq s$;%
    	\item[2. ] there exist constants $\theta_s, \theta_{s,s}$ and $\theta_{s,2s}$ such that
    	\begin{itemize}
    	    \item[a) ]  for every $x\in \X$ with $\abs{\supp(x)}\leq s$,
            \begin{equation*}
            	\left(1 - \theta_s\right)\norm{x}{}^2\leq\norm{Ax}{}^2\leq\left(1+\theta_s\right)\norm{x}{}^2;
            \end{equation*}
            \item[b) ] for every $x,x'\in \X$ with $\abs{\supp(x)}\leq s$, $\abs{\supp(x')}\leq s$ (resp. $\abs{\supp(x')} \leq 2s)$ and $\supp(x)\cap \supp(x')=\emptyset$,
            \begin{equation*}
            		\abs{\langle Ax,Ax'\rangle } \leq \theta_{s, s}\nor{x} \nor{x'}.
            \end{equation*}
            (resp. $\abs{\langle Ax, Ax'\rangle } \leq \theta_{s, 2s}\nor{x} \nor{x'}$).
            \item[c) ] $\theta_s+\theta_{s,s}+\theta_{s,2s} < 1$.
    	\end{itemize}
    \end{itemize}
\end{assumption}
\begin{lemma}[\cite{grasmair2011necessary}, Prop. 5.3]\label{lem:CS}
	Suppose \Cref{ass:CS} holds.
	Then:
	\begin{itemize}
		\item The vector $\xs$ is the unique primal solution of Problem \eqref{eq:exact} with $R(\cdot)=\norm{\cdot}{1}$; namely,
		\begin{equation}
		\argmin_{x\in\X} \left\{\norm{x}{1}: \ \ Ax=\bbs \right\} = \left\{\xs\right\} \enspace .
		\end{equation}

		\item There exists a dual solution $\ys\in \Y$ such that
		\begin{equation*}
			\norm{\ys}{}\leq W_s := \frac{\sqrt{s}}{\sqrt{1-\theta_s}} \frac{\theta_{s,s}}{1-\theta_s-\theta_{s,2s}} \quad \text{and} \quad m\leq M_s: = \frac{\theta_{s,s}}{1-\theta_s-\theta_{s,2s}}<1,	\end{equation*}
		where $m$ is the saturation gap (\Cref{prop:extended_supp})  related to $\ys$.
		\item Let $\X_{\Ss}:=span\left\{e_i: \quad i\in \supp(\xs) \right\}$ and by $i_{\Ss}: \ \X_{\Ss}\to \X$ the identity embedding. Then $A_{\Ss}:=A \circ i_{\Ss}$ is injective with
	\begin{equation}
        \label{QQ} \norm{A_{\Ss}^{-1}}{}\leq Q_s:= \frac{1}{\sqrt{1-\theta_s}}.
	\end{equation}
		\item For every $x\in \X$,
		\begin{equation}\label{ineq_CS}
		\begin{split}
			\norm{x-\xs}{} & \leq Q_s\norm{Ax-\bbs}{} + \frac{1+Q_s \norm{A}{}}{1-M_s} D^{-A^*\ys}\left(x,\xs\right).
		\end{split}
	\end{equation}
	\end{itemize}
\end{lemma}
\ \\
The previous results applied to Tykhonov regularization with $\ell^1$ norm (Lasso), allow to derive explicit regularization results, which we  recall in \Cref{app:sub:tycho_sparse}.
In an a similar fashion, we use these facts for our proposed iterative regularization method instead.

\subsection{Sparse recovery with iterative regularization}\label{ell1itreg}
Combining  \Cref{prop:early_stop}  with \Cref{ass:CS} and the inequality in \eqref{ineq_CS}, we get the following theorem for sparse recovery with $\ell_1$-norm.
In this setting, we are able to get an upper bound for the distance between the iterates and the exact solution.
\begin{theorem}\label{thm:model_recovery}
	Suppose that \Cref{ass:CS} holds.
	Let $\xs \in \X$ be the unique primal solution of the exact problem
	$$\min_{x\in\X} \left\{\norm{x}{1}: \ \ Ax=\bbs \right\} \enspace,$$
	and $\ys\in\Y$ the dual solution given by \Cref{lem:CS}.
	Moreover, under \Cref{ass:TSigm,ass:omega,ass:thetarho}, let $\left(\hat{x}_k,\hat{y}_k\right)$ be the sequence of averaged iterates generated by the primal-dual algorithm \ref{eq:algo} when applied to the inexact problem
	\begin{equation*}
		\min_{x \in \X} \left\{ \nor{x}_1: \ \ Ax = \bb \right\} \enspace.
	\end{equation*}
	Then we have that, for every $k \in \mathbb{N}$,
	\begin{align}\label{eq:ell1}
			\norm{\hat{x}_k-\xs}{}
			& \leq Q_s \norm{A\hat{x}_k-\bbs}{} + \frac{1+Q_s\norm{A}{}}{1-M_s} D^{-A^*\ys}\left(\hat{x}_k,\xs\right) \nonumber \\
			& \leq Q_s \sqrt{ \frac{C_4}{k} + C_5 \delta + C_6 \delta^2 +C_7 \delta^2 k }
			    + \frac{1+Q_s\norm{A}{}}{1-M_s} \left[  \frac{C_1}{k} +C_2 \delta + C_3 \delta^2  k  \right]
	\end{align}
\end{theorem}
\begin{remark}[Dependence on initialization]
Notice that the bound \eqref{eq:ell1} depends on the initialization $z_0$, through $V(\zs - z_0)$ in the $C_i$'s, see \Cref{app:pf_earlystop}.
Yet, using the initialization $z_0 = 0$, we can bound the term $V(\zs - z_0)$ by  quantities that do not involve the unknown solution $\zs$:
	\begin{equation}\label{eq:aux}
		\begin{split}
			V(z_0-\zs) & = \frac{1}{2\tau} \norm{\xs}{}^2 +\frac{1}{2\sigma} \norm{\ys}{}^2 \\
		\eqref{QQ} \ 	& \leq \frac{\norm{A_{\Ss}\xs}{}^2 }{2\tau Q_s^2} +\frac{W_s^2}{2\sigma}\\
			& = \frac{\norm{\bbs}{}^2}{2\tau Q_s^2} +\frac{W_s^2}{2\sigma}\\
			& \leq \frac{2\norm{\bbs-\bb}{}^2+2\norm{\bb}{}^2}{2\tau Q_s^2}  +\frac{W_s^2}{2\sigma}\\
			& \leq \frac{\delta^2+\norm{\bb}{}^2}{\tau Q_s^2}  +\frac{W_s^2}{2\sigma}.
		\end{split}
	\end{equation}
	For comparison with Tykhonov explicit regularization, we recall a result from \cite{grasmair2011necessary} (see \Cref{cor:Tyk} in the Appendix for the precise statement). Under \Cref{ass:CS} and for $\alpha>0$, let
	\begin{equation}\label{TykProbl}
		x_\alpha \in \argmin_{x\in\X} \left\{\nor{Ax - b^\delta}^2 + \alpha \nor{x}_1 \right\}.
	\end{equation}
Then, defining $D:=\left(1+Q_s\norm{A}{}\right)/\left(1-M_s\right)$, %
\begin{align*}
 & \norm{x_{\alpha}-\xs}{}
     \leq \left(Q_s W_s+D W_s^2/4\right)\alpha + \left(Q_s+D W_s\right) \delta + D \frac{\delta^2}{\alpha}.
\end{align*}
In particular, in the case of Tykhonov regularization, the upper bound does not depend on the magnitude of the exact or noisy data. %
	On the other hand, from \eqref{eq:ell1} and \eqref{eq:aux} we do not get a bound independent from the magnitude of $\bb$.  However, in general, the exact solution of Tykhonov problem is not available in closed form and must be approximated numerically by some iterative algorithm; the main examples are forward-backward (also called ISTA in this context) or accelerated forward-backward (FISTA). When these methods are applied to Tykhonov problem, the distance between the iterate and the solution depends indeed on the initialization and on the magnitude of the data, as for the proposed primal-dual algorithm.
\end{remark}

\section{Unfeasible case: convergence and stability with respect to a normal solution}
\label{sec:unfeas}
In this section, we consider the case where the ideal problem is not feasible, i.e. the linear equation $Ax=b^\star$ does not have a solution.
We show that to provide convergence and stability results for \Cref{eq:algo}  it is enough to assume that the normal equation $A^*Ax=A^*b^\star$  has a solution.
Indeed, this is the classical setting in ill-posed inverse problems \cite{engl1996regularization},  but rarely considered in the context of iterative regularization beyond Hilbertian norms. This generalization is especially relevant for infinite dimensional problems.

In the first part of this section we focus on convergence and we refer to a generic data $b \in \Y$ on which the algorithm is run, as the presented results can be applied both to the exact and the inexact data.
We denote the set of primal solutions with data $b$ simply as $\PP$, the one of dual solutions as $\DD$, the Lagrangian as $\LL$ and the set of saddle-points as $\mathcal{S}$.
Let $(x_k,y_k)$ be the sequence generated by the primal-dual algorithm in \Cref{eq:algo} with data $b$.
First we have the following result, showing that every weak cluster point of the averaged iterates is a saddle-point. Therefore, if there are no saddle-points, the iterates must diverge.
\begin{restatable}{corollary}{unfeas}\label{lem:unfeas}
        Let \Cref{ass:gen} hold.
        Let $(x_k,y_k)$ be the sequence generated by \Cref{eq:algo} with data $b$ under \Cref{ass:TSigm}, \Cref{ass:omega} and summable error ($(\varepsilon_k)\in\ell^1$).
        Denote by $(\hat{x}_k,\hat{y}_k)$ the averaged iterates.
        Then, every weak cluster point of $(\hat{x}_k,\hat{y}_k)$ belongs to $\mathcal{S}$. %
         In particular, if $\Ss=\emptyset$,  then the primal-dual sequence $(\hat{x}_k,\hat{y}_k)$ diverges: $\norm{(\hat{x}_k,\hat{y}_k)}{}\to +\infty$.
\end{restatable}

The proof can be found in \Cref{app:unfeas_1}. The above result ensures that every weak cluster point of the averaged sequence belongs to $\mathcal{S}$ (and so to $\PP\times\DD$ by \Cref{remark:cond}).
Moreover, if there are no primal solutions ($\PP=\emptyset$), then $\PP\times\DD=\emptyset$, $\mathcal{S}=\emptyset$ and so the joint sequence $(\hat{x}_k,\hat{y}_k)$ diverges.
Yet we are mainly interested in the primal variable, which may still converge while  $\norm{(\hat{x}_k,\hat{y}_k)}{}\to+\infty$.
In the sequel, we show sufficient conditions for the averaged primal iterates to converge even when $\PP = \emptyset$. For this purpose, we introduce the feasible set and the \emph{normal} feasible set as
\begin{equation}
    \CC:=\left\{x\in\X: \ \ Ax=b\right\} \quad \text{and} \quad \tilde{\CC}:=\left\{x\in\X: \ \ A^*Ax=A^*b\right\} \enspace.
\end{equation}
It is clear that $\CC\subseteq\tilde{\CC}$.
Moreover, $\CC \neq \emptyset$ implies that $\tilde{\CC}=\CC$.
Indeed, let $u\in\tilde{\CC}$ and pick any $x\in\CC$.
Then $A^*Au=A^*b$, $Ax=b$ and $u-x\in N(A^*A)=N(A)$.
Thus $Au=Ax=b$; and so $u\in\CC$.
\\In addition to the normal feasible set, define also the \emph{normal} primal problem, its dual, and the \emph{normal} Lagrangian as
\begin{align}\label{normprob}
	\tilde{\PP} &:= \argmin_{x \in \X } \left\{ R(x)+F(x): \ \ A^*A x = A^* b \right\}, \\
	\tilde{\DD} &:= \argmin_{v \in \X } \left\{\left[R+F\right]^*(-A^*A v) + \langle A^*b, v \rangle \right\},\\
    \tilde{\LL}(x, v) &:= R(x) + F(x) + \langle v, A^*Ax - A^*b \rangle.
\end{align}
From $\CC\neq \emptyset \implies \tilde{\CC}=\CC$, we have $\CC \neq \emptyset \implies \tilde{\PP}=\PP.$
But it may happen that $\CC=\emptyset$ and $\tilde{\CC}\neq \emptyset$; and, consequently,  that there are no primal solutions ($\PP = \emptyset$) but there are normal primal solutions ($\tilde{\PP}\neq \emptyset$).
Thus in the next results, considering the case $\PP = \emptyset$ and $\tilde{\PP}\neq \emptyset$, we show convergence and stability with respect to a normal solution. More precisely,
\begin{itemize}
	\item in \Cref{th:unfeas_norm}, we show sufficient conditions to get convergence of the averaged primal sequence to a point in $\tilde{\PP}$ even though $\PP=\emptyset$;
	\item in \cref{th:unfeas_stab}, we get stability and early-stopping results analogous to the ones in \Cref{prop:early_stop} but with respect to any normal solution.
\end{itemize}

For simplicity, in the remainder of this section we include neither the preconditioning nor the error in the proximal-operator, setting $T=\tau \Id$, $\Sigma = \sigma \Id$ and $\varepsilon_k=0$ for every $k\in\N$.
Then our algorithm can be written as: given $x_0, y_{-1}$ and setting $y_0=y_{-1}+\sigma(Ax_0-b)$, for every $k\in\mathbb{N}$,
\begin{align}\label{algo:easy}
    \begin{cases}
	\tilde{y}_{k} = 2 y_k - y_{k - 1} \enspace, \\
	x_{k+1} = \prox_{\tau R}(x_k - \tau \nabla F(x_k) - \tau A^*\tilde y_{k}) \enspace,\\
	y_{k + 1} = y_k + \sigma \left(Ax_{k + 1} - b\right) \enspace.
    \end{cases}
\end{align}
Assume that $\tilde{\CC}\neq \emptyset$.
Let $x^b\in\tilde{\mathcal{C}}$ (meaning that $A^*Ax^b=A^*b$) and let $S:=(A^*A)^{\frac{1}{2}}$.
The normal problem \eqref{normprob} then can be rewritten as:
\begin{equation}\label{leastsquares}
    \tilde{\PP}=\argmin_{x\in\X} \left\{ R(x)+F(x): \ \ Sx=Sx^b \right\} \enspace.
\end{equation}
Indeed, $N(S)=N(S^*S)=N(A^*A)$ and $A^*Ax=A^*b=A^*Ax^b \Leftrightarrow x-x^b\in N(A^*A) = N(S)$.

In \Cref{lem:unfeas_sameit}, we show that, under mild conditions, the primal variable generated by the algorithm, when applied to problem $\PP$, is an instance of the same procedure but applied to the normal problem $\tilde{\mathcal{P}}$ in the form \eqref{leastsquares}.

\begin{restatable}{lemma}{unfeassameit}\label{lem:unfeas_sameit}
	Let \Cref{ass:gen} hold.
	Assume that $\tilde{\mathcal{C}}\neq \emptyset$. %
	Let $\seq{x_k}$ be the primal sequence generated by algorithm \eqref{algo:easy}; namely, with $T=\tau \Id$, $\Sigma = \sigma \Id$, $\varepsilon_k=0$ for every $k\in\N$ and $y_0=y_{-1}+\sigma(Ax_0-b)$. %
	Then, there exists a primal sequence $\seq{u_k}$ generated by the same procedure but applied to problem $\tilde{\PP}$ (as stated in \eqref{leastsquares}) such that $x_k=u_k$ for every $k\in\N$.
\end{restatable}
The proof can be found in \Cref{app:unfeas_2}. We are now ready to state the two main results of this section. The first one shows weak convergence of the averaged primal iterate of the algorithm, when applied to $\mathcal{P}$, to a solution of the normal problem $\tilde{\mathcal{P}}$.
\begin{restatable}{theorem}{unfeasstocazzo}\label{th:unfeas_norm}
	Let \Cref{ass:gen} hold.
	Assume that $\tilde{\PP}$ (as stated in \ref{normprob}) admits a saddle-point; namely, that there exists a pair $(\tilde{x},\tilde{v})\in\X \times \X$ such that
\begin{equation}\label{ciao}
    \begin{cases}
    -A^*A\tilde{v}\in\partial R(\tilde{x})+\nabla F(\tilde{x}) \enspace, \\
    A^*A\tilde{x}=A^*b \enspace.
    \end{cases}
\end{equation}
    Let $(x_k,y_k)$ be the sequence generated by \Cref{algo:easy}, namely with initialization  $y_0=y_{-1}+\sigma(Ax_0-b)$, and under \Cref{ass:omega}. %
    Denote by $(\hat{x}_k)$ the averaged primal iterates.
    Then there exists $\tilde{x}_{\infty}\in\tilde{\PP}$ such that $\hat{x}_k\rightharpoonup \tilde{x}_{\infty}$.
    Moreover, if $\PP = \emptyset$, then $\hat{y}_k$ diverges.
\end{restatable}

The proof can be found in \Cref{app:unfeas_norm}.
Since we assume that the normal problem has a saddle point, a priori we could apply the primal-dual algorithm directly to the normal problem $\tilde{\mathcal{P}}$ and therefore with $A^*A$ in place of $A$. To fix the ideas, consider the final dimensional setting, in which $A\in\mathbb{R}^{n\times d}$. If $d>n$, as is usual in compressed sensing, working with the matrix $A^*A$ can be disadvantageous.
\ \\
Two questions remain open from our previous analysis, that we leave as future work. Consider for simplicity the case $F=0$. From the definition of the primal iterates in the proposed algorithm and the properties of the $prox$ operator, we know that, if the domain of $R$ is bounded, then the primal iterates remain bounded. Suppose that the normal equation has solutions, namely $\tilde{\mathcal{C}}\neq \emptyset$. If the domain of $R$ does not intersect  $\tilde{\mathcal{C}}$, we expect - but we could not prove - that the primal iterates of the algorithm converge to an element in
    $$\argmin_{x\in dom(F)} \ \inf_{y\in\tilde{\mathcal{C}}} \norm{x-y}{}.$$
    On the other hand, now suppose - for instance - that the function $R$ has full domain. We have seen that if the normal problem admits a saddle-point, then the averaged primal sequence converges to an element in $\tilde{\mathcal{P}}$ (see \Cref{th:unfeas_norm}). We expect that, on the contrary, the absence of solution for the primal normal problem (for instance, if $\tilde{\mathcal{C}}=\emptyset$) implies divergence of the primal iterates. This is the case of the example discussed in \Cref{sub:divergence}, but we could not prove it in general.\\
\ \\
To conclude this section, we show a stability result for the iterates generated by the algorithm on the noisy data with respect to any saddle-point of the exact \emph{normal} problem. For this theorem we come back to the separated notation $b^{\star}$ for the exact data and $b^{\delta}$ for the noisy one, while we keep the symbol tilde for normal problems and solutions; for instance, $\tilde{\PP}^{\star}$ will denote the exact normal primal problem, as stated for instance in \Cref{normprob} but with data $b^{\star}$.
\begin{restatable}{theorem}{unfeasstab}\label{th:unfeas_stab}
	Let \Cref{ass:gen} hold
	and suppose that there exists a pair $(\tilde{x},\tilde{v})\in\X \times \X$ such that
\begin{equation}\label{optcond}
    \begin{cases}
    -A^*A\tilde{v}\in\partial R(\tilde{x})+\nabla F(\tilde{x}) \enspace, \\
    A^*A\tilde{x}=A^*b^{\star} \enspace
    \end{cases}
\end{equation}
(namely, a saddle-point for the normal exact problem $\tilde{\PP}^{\star}$). Let $b^{\delta}\in\Y$ be a noisy data such that $\norm{b^{\delta}-b^{\star}}{}\leq \delta$ for some $\delta\geq 0$. Moreover, suppose that $\tilde{\mathcal{C}}^\delta\neq \emptyset$; namely, that there exists $x^{\delta}\in\X$ such that $A^*Ax^{\delta}=A^*b^{\delta}$. Let \Cref{ass:omega} and \Cref{ass:thetarho} hold and $(x_k,y_k)$ be the sequence generated by the algorithm \Cref{algo:easy} on the noisy data $b^{\delta}$; namely, for the initialization $y_0=y_{-1}+\sigma(Ax_0-b^{\delta})$,
    \begin{align*}
    \begin{cases}
	\tilde{y}_{k} = 2 y_k - y_{k - 1} \enspace, \\
	x_{k+1} = \prox^{}_{\tau R}(x_k - \tau\nabla F(x_k) - \tau A^*\tilde y_{k})  \enspace,\\
	y_{k + 1} = y_k + \sigma \left(Ax_{k + 1} - b^{\delta}\right).
    \end{cases}
\end{align*}
    Denote by $(\hat{x}_k)$ the averaged primal iterates. Then,
    \begin{equation*}
    D^{-A^*A\tilde{v}}(\hat{x}_k,\tilde{x}) \leq \frac{C_1}{k} +C_2 \delta + C_4 \delta^2  k
\end{equation*}
and
\begin{equation*}
    \norm{A^*A\hat{x}_k-A^*b^{\star}}{}^2\leq \norm{S}{}\left[\frac{C_5}{k} + C_6 \delta+  C_8 \delta^2 k  + C_9 \delta^2 \right] ,
\end{equation*}
where the constants involved in the bounds are specified in the proof.
\end{restatable}
The proof can be found in \Cref{app:unfeas_3}.
\begin{remark}
We think that the assumption $\tilde{\mathcal{C}}^\delta\neq \emptyset$ is a technical byproduct of our analysis (we need to assume it to use \Cref{lem:unfeas_sameit}), but not necessary in order to get the results in \Cref{th:unfeas_stab}.
\end{remark}

\begin{example}
It is easy to find an example explaining the meaning and the importance of the previous result. Consider the following setting in $\X=\R^2$.
Let the inexact linear system $Ax=b^\delta$ identify a line on the plane and let $R: \ \R^2 \to \R$ be a convex and lower-semicontinuous function that is an exponential when restricted to the inexact constraint $\mathcal{C}^\delta$. Then, $\tilde{\mathcal{C}}^\delta=\mathcal{C}^\delta\neq \emptyset$ but $\tilde{\PP}^\delta=\PP^\delta= \emptyset$. In particular we are in a case of severe instability: the averaged primal iterates $(\hat{x}_k)$, generated by the algorithm when applied to problem $\PP^\delta=\emptyset$, may diverge. Now consider the two following scenarios.
\begin{itemize}
    \item Let the exact linear system $Ax=b^\star$ identify a line in $\R^2$ (parallel to $\mathcal{C}^\delta$) and let $R: \ \R^2 \to \R$ be coercive on the exact constraint $\mathcal{C}^\star$. Then the primal exact problem admits minimizers ($\mathcal{P}^\star\neq\emptyset$), while the noisy one does not have solutions even if it is feasible. In this setting, the assumptions of \Cref{prop:early_stop} hold and thus our early-stopping bounds guarantee an efficient way to find a stable solution.
    \item Now suppose that the exact linear system $Ax=b^\star$ does not admit solutions ($b^\star\notin R(A)$) and let the exact normal system $A^*Ax=A^*b^\star$ identify a line in $\R^2$. Moreover, similarly to the previous example, let $R: \ \R^2 \to \R$ be coercive on the exact normal constraint $\tilde{\mathcal{C}}^\star$. The primal exact problem does not admit feasible points and so neither minimizers ($\mathcal{P}^\star=\emptyset$). Then, in this case, the assumptions in \Cref{prop:early_stop} are not verified. On the other hand, the exact normal problem has solutions ($\tilde{\mathcal{P}}^\star=\emptyset$) and $\tilde{\mathcal{C}}^\delta\neq \emptyset$, so we still can apply \Cref{th:unfeas_stab} to get an a similar early-stopping result, but with respect to any exact normal solution.
\end{itemize}
\end{example}
\section{Experiments}
\label{sec:experiments}

A high quality Python package implementing our iterative regularization approach, with reproducible experiments, is available at \url{https://lcsl.github.io/iterreg}.

\subsection{Sparse recovery with the $\ell_1$ norm}
\label{sub:xp_vs_lasso}

First we illustrate numerically the results of \Cref{sec:ell1} ($R(\cdot) = \nor{\cdot}_1$, $F=0$) on both real data and simulations.
The simulated data is generated as $b^\delta = b^\star + \epsilon = A \bar{x} + \epsilon$.
The design matrix $A$ has Gaussian entries with a Toeplitz correlation structure (correlation between columns $i$ and $j$ is $\rho ^{|i-j|}$ for $\rho \in [0, 1[$; as $\rho$ approaches 1, the problem becomes more and more difficult).
The noise vector $\epsilon$ has i.i.d. Gaussian entries, with standard deviation scaled to control the signal-to-noise ratio (SNR), defined as $\nor{A\bar x} / \nor{\epsilon}$.
The true parameter vector $\bar{x}$ has 10 \% non zero entries set to 1 ; note that the noiseless solution $\xs$ is not necessarily $\bar x$ -- in particular the $\ell_0$ and $\ell_1$ solutions tend to differ if the feature correlation parameter $\rho$ is too high or if the sparsity of $\bar x$ is not low enough.
In Algorithm \eqref{eq:algo}, unless specified otherwise, we use exact prox ($\varepsilon_k = 0$), as well as scalar preconditioners $T = \tau \Id$ and $\Sigma = \sigma \Id$.

The explicit, Tykhonov regularization competitor in this case is the Lasso.

\paragraph{Datadriven choice of stepsize $\sigma$.}

A key distinction between iterative and Tykhonov regularization is that our iterative approach produces discrete iterates, while the Tykhonov path can be discretized with arbitrary precision.
Hence, our algorithm could converge too fast to the noisy solution, preventing us from finding a good early stopped iterate.
Fortunately, it is possible to act on the dual stepsize $\sigma$ so that the iterates remain sparse in the beginning (in the same way as, for the Lasso, the solutions are sparse for large regularization strength $\lambda$).
On \Cref{fig:datadriven} we illustrate multiple choices for $\sigma$, keeping $\sigma \tau$ equal to $ 0.99/\nor{A}^2$: $\sigma \in \{\tau, \tau / 100, 1 / \Vert A^* b^\delta \Vert_\infty, \tau / 10000\}$.
The order of magnitude $\sigma = 1/\Vert A^* b^\delta \Vert_\infty$ is reversed engineered from the first iterations of \eqref{eq:algo} with $x_0 = 0$, $y_{-1} = y_0 = 0$, yielding $y_1 = - \sigma b^\delta$ and ensuring that $x_2 = \prox_{\tau \Vert \cdot \Vert_1}{(2\tau\sigma A^* b^\delta)}$ remains sparse enough.
\\The performance of iterative regularization is measured by the F1 score between the support of the iterates and the support of the true parameters, $\bar x$.
As visible on \Cref{fig:datadriven}, the higher $\sigma$, the faster the primal iterates $x_k$ become dense, thus overestimating the support of $\bar x$.
From the figure, one can see that the datadriven choice of $\sigma$ provides a good balance between quality of the regularization (it reaches the highest F1 score) and convergence speed (optimal score reached after 15 iterations only).

\begin{figure}[t]
    \centering
    \includegraphics[width=0.8\linewidth]{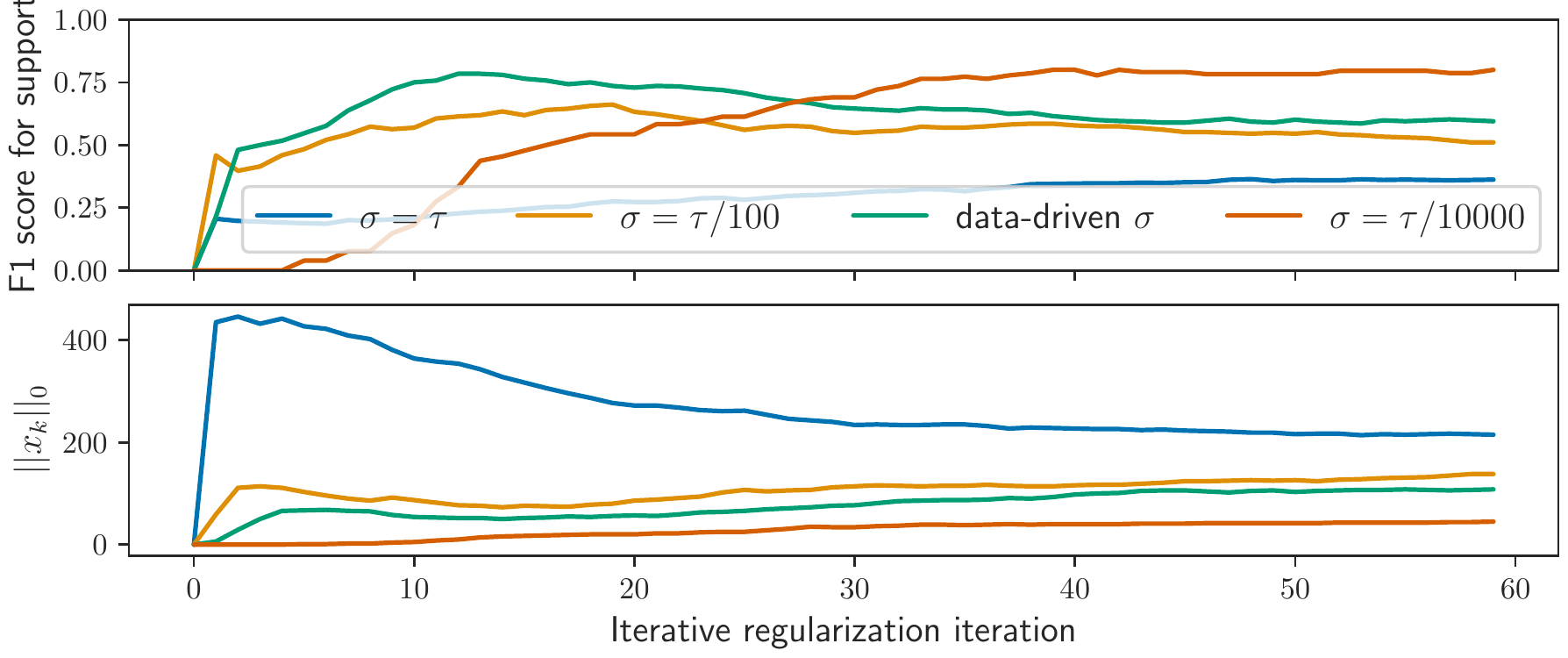}
    \caption{To maintain sparsity in the early iterates, it is important to set $\sigma$ correctly:
    if it is too big, the iterates are dense too quickly (blue curve); if it is too low, convergence is too slow (red). Our datadriven choice behaves well: the iterates sparsity increases steadily, and they reach the highest F1 score. $(n, d, \rho) = (200, 500, 0.2)$, $\nor{A \bar{x}} / \nor{\epsilon} = 10$.}
    \label{fig:datadriven}
\end{figure}

\paragraph{Comparison with the Lasso  on simulations.}
In this experiment, we compare the support recovery performance to that of the Lasso.
In order to have a ground truth available, we use a simulated setup.
The data for this experiment has 1000 samples and 2000 features.
The performance of iterative and Tykhonov regularization is evaluated with the F1 score for support estimation, and normalized mean squared error on left out data (250 additional samples) for prediction, $\Vert b^{\delta, \, \mathrm{test}} - A^{\mathrm{test}} \Vert^2 / \Vert b^{\delta, \, \mathrm{test}} \Vert^2$.
We study two scenarios: an ``easy'' one ($\mathrm{SNR} = 5$, low feature correlation factor $\rho=0.2$) and a more challenging one ($\mathrm{SNR} = 3$, $\rho=0.8$).
On \Cref{fig:lasso_simu}, one can see that the estimation and prediction performances are comparable between iterative regularization and explicit regularization, illustrating the numerical guarantees of \Cref{sec:ell1}.

\begin{figure}[t]
    \centering
    \includegraphics[width=0.45\linewidth]{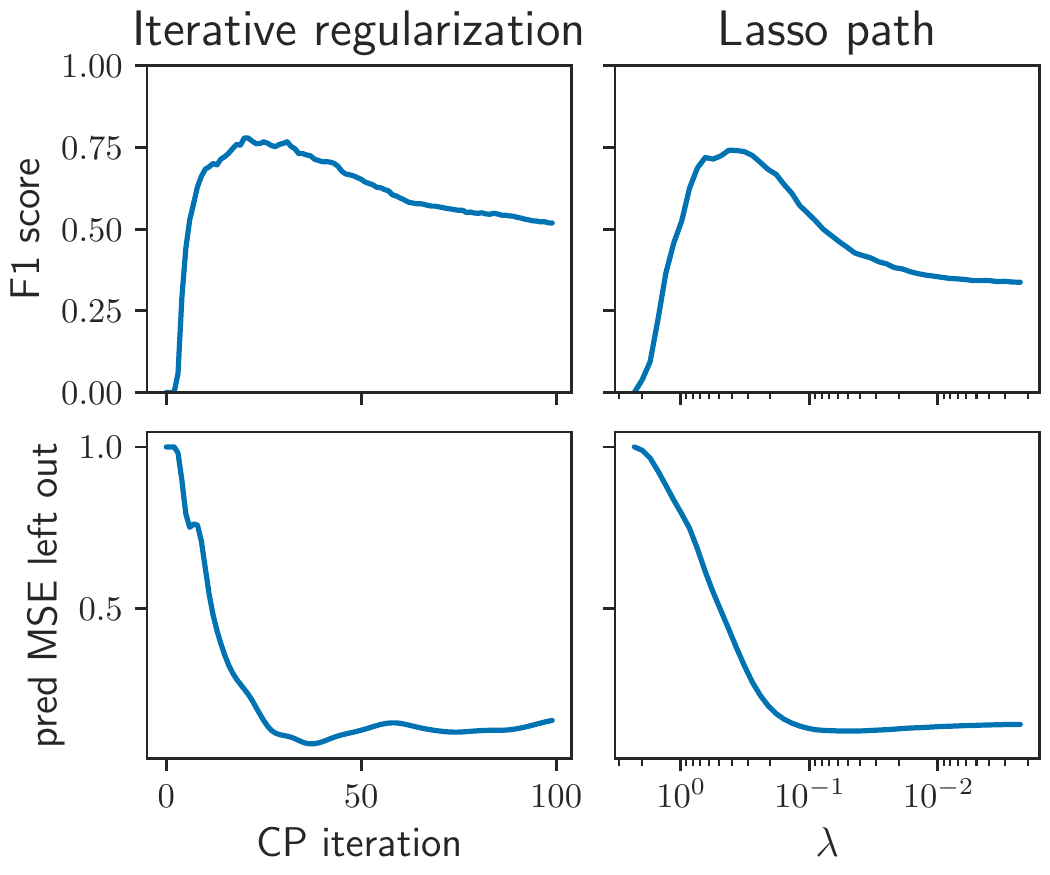}
    \hfill
    \includegraphics[width=0.45\linewidth]{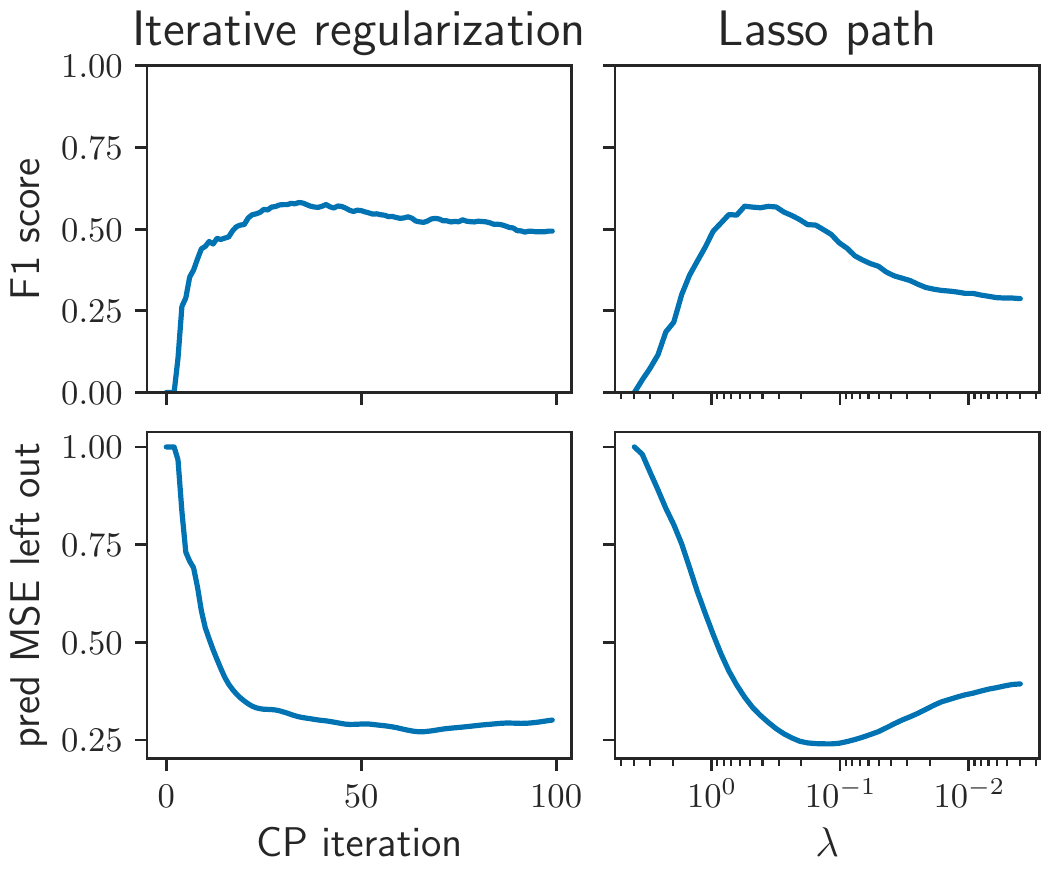}
    \caption{Comparison of estimation and prediction performances of iterative and Tykhonov regularization for sparse recovery.
    Left: feature correlation factor $\rho = 0.2$, $\mathrm{SNR} = 5$. Right: correlation factor $\rho = 0.8, \mathrm{SNR} = 3$.
    In both scenarios, iterative regularization attains performances similar to explicit regularization, but in a few iterations.}
    \label{fig:lasso_simu}
\end{figure}

\paragraph{Timing comparison with the Lasso on real data.}

Finally, we benchmark our approach on real data, where the true support is unknown and the best model must be selected by cross validation

In \Cref{fig:regpath_optpath}, we compare the quality of solutions obtained by iterative regularization and explicit regularization.
The dataset for this experiment is \emph{rcv1} from the LIBSVM package\footnote{\url{https://github.com/mathurinm/libsvmdata}}, for which $(n, d) = (\num{20242}, \num{19959})$.
In order to select the best regularization strength for each approach (iteration or value of $\lambda$), we use the prediction mean squared error with 4-fold cross validation: the data $(A, b^\delta)$ is split in 4 folds and each method is run 4 times on 3 folds, while the MSE is computed on the remaining, unseen fold (dashed colored lines).
The MSE is then averaged across folds (thick black line), and the best iteration/$\lambda$ is determined by its minimum.
Note that this approach does not rely on the knowledge of the true parameters $\bar x$ and is thus the one we advocate to use to determine the optimal stopping time in practice.

To solve the Lasso, we use the state-of-the-art solver \texttt{celer} \cite{massias2019dual}, based on coordinate descent, an active set strategy and Anderson acceleration.
Extensive validation in \cite{massias2019dual} showed that this algorithm was currently the fastest one available to solve the Lasso.
Warm-start is used along the path: the solution for the previous $\lambda$ is used as initialization for the next one.
With all these improvements over a basic forward-backward solver, the time to compute the best solution (the path up to the best $\lambda$, if it were known in advance) is 125 seconds.
This is because 69 Lasso problems must be solved (the optimal $\lambda$ is the $69$-th on the grid), each one being increasingly difficult as $\lambda$ decreases.

On the contrary, iterative regularization finds its optimal solution along the optimization path in 2.5 s.
The cost of each iteration is $\mathcal{O}(nd)$, making the algorithm very fast.
One can see that in terms of prediction error on left-out data (4-fold cross validation being used to determine both the best $\lambda$ for the Lasso and the best early stopping for our approach), both methods reach a similar performance, with a best average MSE around 0.2.
In addition, using our proposed datadriven stepsize, we obtain a sparser solution than the Lasso: ours has \num{1583} non zeros entries, while the optimal Lasso one has \num{2820}.

\begin{figure}[t]
    \centering
    \includegraphics[width=\linewidth]{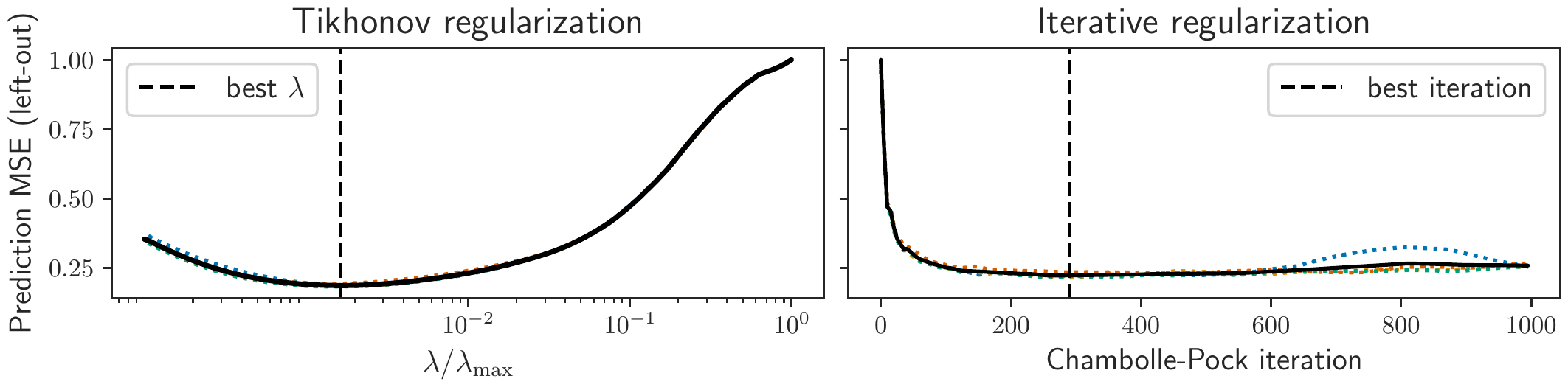}
    \caption{Comparison of Tikhonov regularization and iterative regularization. The figure of merit is 4-fold cross validation prediction error.
    Both methods reach similar lowest prediction errors (left: 0.195, right: 0.21)  while the iterative approach is much faster (2.5 s vs. 125 s).
    }
    \label{fig:regpath_optpath}
\end{figure}

\subsection{Preconditioning}
In this experiment we highlight the usefulness of a preconditioning.
We consider two diagonal preconditioners, following \cite{pock2011diagonal}: $T = \theta \diag(||A_{:1}||^2, \ldots, ||A_{:d}||^2)$ and $\Sigma = \tfrac{1}{\theta} \diag(\nor{A_{1:}}_0, \ldots, \nor{A_{n:}}_0) = \frac{d}{\theta} \Id$.
The scaling factor $\theta$ is set to get $\sigma$ as in the datadriven choice detailed above.
This choice of $T$ and $\Sigma$ satisfies $\tau_M \sigma_M \leq 1 / \nor{A}^2$ \cite[Lemma 2]{pock2011diagonal}.
The design matrix $A$ is generated as in \Cref{sub:xp_vs_lasso}, but each column is then scaled by a uniform random number between 1 and 5, resulting in different column norms and thus in $T$ being different from a scalar matrix.
On \Cref{fig:precond}, one an see that using coordinate-wise stepsizes through the use of $T$ in the update of the primal variable, is beneficial for iterative regularization as a higher F1 score is reached.

\begin{figure}[t]
    \centering
    \includegraphics[width=0.6\linewidth]{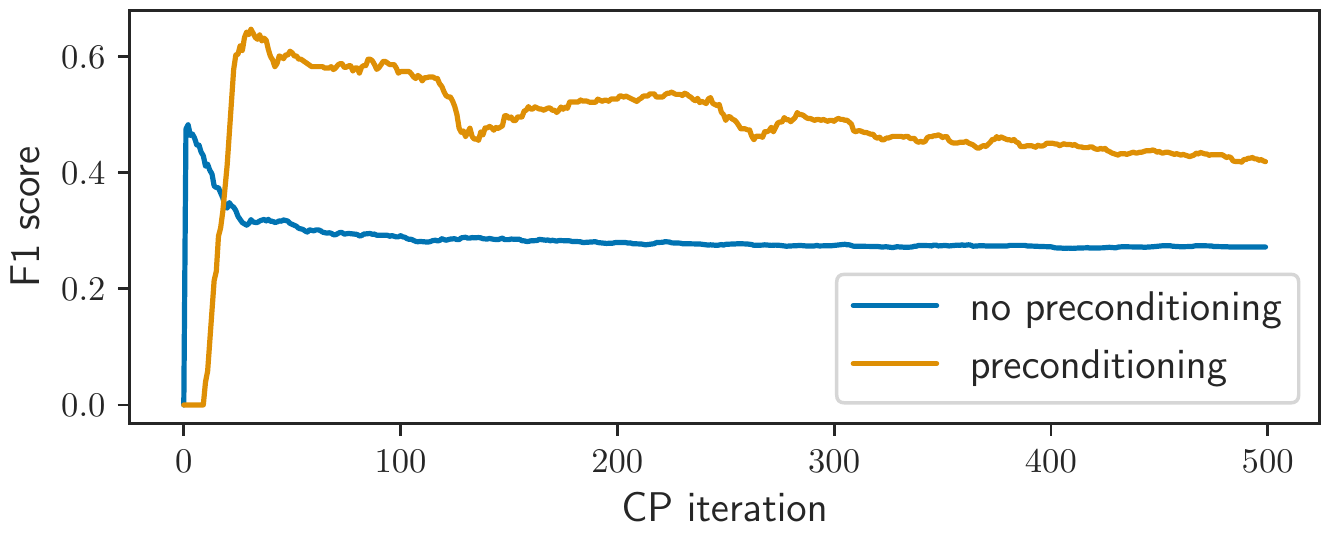}
    \caption{Benefit of preconditioning for sparse recovery on an unnormalized simulated dataset. $(n d) = (500, 1000)$.}
    \label{fig:precond}
\end{figure}

\subsection{Low rank matrix completion}

In this experiment we highlight the versatility of our approach, considering the matrix completion setting of \Cref{ex:low_rank}.
The goal is to recover a low-rank matrix from the noisy observation of a subset of its entries.
Both Hilbert spaces $\X$ and $\Y$ are taken equal to $\bbR^{d \times d}$, and we use upper case letters $X$ and $B$ to denote the primal variable and the observations.
The true matrix to recover is chosen as $B^\star = UV^\top$ where $U, V \in \bbR^{d \times 5}$ have i.i.d. normal entries.
In order to get meaningful values for $\delta$, we scale $B^\star$ so that it has a norm equal to $20$.
Finally, for a range of values of $\delta$, various $B^\delta$ are obtained by adding scaled random Gaussian noise to the observed entries of $B^\star$
We choose to hide 80 \% of entries of $B^\delta$, uniformly sampled.
The matrix $A$ corresponds to the masking operator; we have $\nor{A}_2 = 1$ and thus use $\sigma = \tau = 0.99$.
We tune the parameter $\sigma$ similarly to the $\ell_1$ case, taking $\sigma = 1 / \nor{A^*B^\delta}_{2}$.
\Cref{fig:low_rank} highlights the semiconvergence behavior exploited by iterative regularization: the iterates produced by \eqref{eq:algo} first get closer to the noiseless solution, before converging to the noisy solution.
Early-stopping the iterate at a correct iteration is thus beneficial.

\begin{figure}[t]
    \centering
    \includegraphics[width=0.45\linewidth]{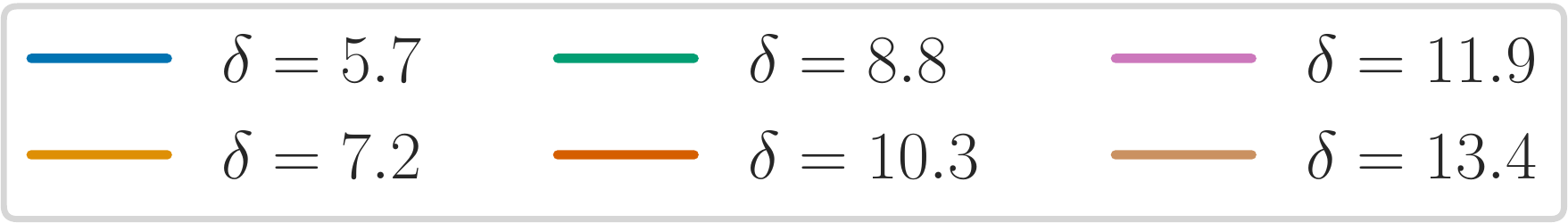} \\
    \includegraphics[width=0.45\linewidth]{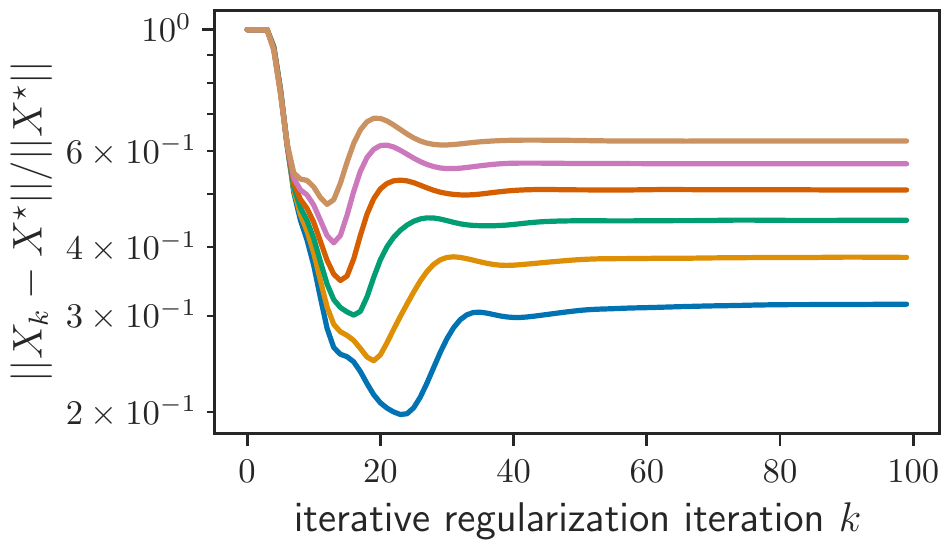}
    \hfill
    \includegraphics[width=0.45\linewidth]{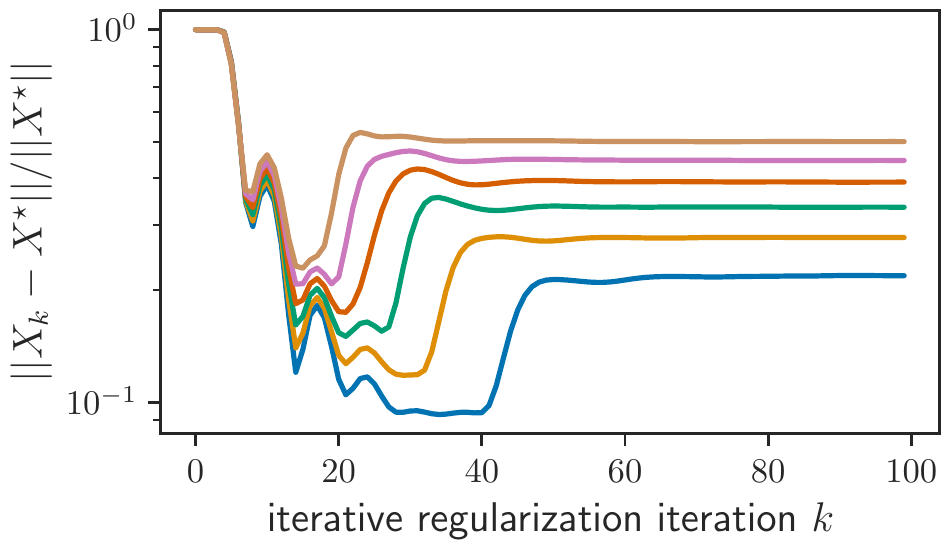}
    \caption{Semiconvergence of iterates for the low rank matrix completion problem, in dimension $200 \times 200$ (left) and $500 \times 500$ (right).
    The iterates first get close to the noiseless solution, before converging to the noisy solution.}
    \label{fig:low_rank}
\end{figure}

\section{Conclusion}
In this work, we have considered the problem of designing iterative regularization  algorithms for bias described by a wide class of  convex functionals.
We proposed and study an iterative regularization method based on a primal-dual approach of which we characterize convergence and especially stability in the presence of noisy data. This latter results allow to derive and early stopping procedure and corresponding error bounds, comparable with those obtainable with variational regularization techniques. Empirical results complement and confirm our theoretical findings, showing that iterative regularization can be at the same time accurate and efficient.

A number of research directions remains unexplored. For example it would be interesting to consider stochastic gradient approaches, that often results in further efficiency improvement. It would also be interesting to extend the considered model to account for other form of noise/errors, including data models  in machine learning, but also considering other, possibly non convex,  penalties. Finally, it would be interesting to consider nonlinear models, and in particular compositional models such as those defining neural networks.

\clearpage
\appendix

\section{Preliminary lemmas}
\label{app:sec:prelim}
\begin{lemma}[{\cite[Lemma 2]{schmidt2011convergence}}]\label{lem:u_n_upper_bound}
	Assume that $(u_j)$ is a non-negative sequence, $(S_j)$ is a non-decreasing sequence with $S_0\geq u_0^2$ and $\lambda\geq 0$ such that, for every $j\in\N$,
	\begin{equation}
		u_j^2\leq S_j+\lambda\sum_{i=1}^{j} u_i \enspace.
	\end{equation}
Then, for every $j\in\N$,
	\begin{equation}
	u_j\leq \frac{\lambda j}{2}+\sqrt{S_j+\left(\frac{\lambda j}{2}\right)^2} \enspace.
\end{equation}
\end{lemma}
\begin{lemma}[{Descent lemma, \cite[Thm 18.15 (iii)]{Bauschke_Combettes11}}]\label{desclemm}Let $f: \X \to \R$ be Fréchet differentiable with $L$-Lipschitz continuous gradient. Then, for every $x$ and $y\in\X$,
	\begin{equation}
		f(y)\leq f(x)+\langle \nabla f(x),y-x\rangle +\frac{L}{2}\norm{y-x}{}^2 \enspace.
	\end{equation}
\end{lemma}

\begin{lemma}\label{proxnonexp}
Let $\mathcal{Z}$ denote $\X$ or $\Y$ and $U$ denote $T$ or $\Sigma$ accordingly.
Let $f \in \Gamma_0(\mathcal{Z})$ and $\varepsilon \geq 0$.
It follows easily from the definition of the $\varepsilon$-subdifferential that if $a, b \in \mathcal{Z}$ satisfy
\begin{equation}
    U^{-1}\left(a - b\right) \in \partial_{\varepsilon} f(b)\enspace,
\end{equation}
then, for every $c \in \mathcal{Z}$,
\begin{equation}
	f(b) - f(c) + \frac{1}{2}\nor{b - c}^2_U - \frac{1}{2}\nor{a - b}^2_U + \frac{1}{2}\nor{b - a}^2_U  \leq \varepsilon \enspace.
\end{equation}
\end{lemma}
\subsection{Primal-dual estimates}

\begin{lemma}[One step estimate]\label{lem:onestep}
	Let \Cref{ass:gen} hold. Let $(x_k,y_k)$ be the sequence generated by iterations \eqref{eq:algo} under \Cref{ass:TSigm}.
	Then, for any $z=(x, y)\in\X \times\Y$ and for any $k\in\N$, with $V(z) := \frac{1}{2}\nor{x}_T^2 + \frac{1}{2} \nor{y}_\Sigma^2$,
	\begin{equation}\label{mainonestep}
		\begin{split}
			& V(z_{k + 1} - z) - V(z_k - z) + \frac{1-\tau_M L}{2\tau_M}\norm{x_{k+1}-x_k}{}^2+ \frac{1}{2}\norm{y_{k+1}-y_k}{\Sigma}^2\\
			& \hspace{3cm} + \left[\LLd(x_{k + 1}, y) - \LLd(x, y_{k + 1})\right] + \langle y_{k + 1} - \tilde y_k, A\left(x -  x_{k + 1}\right)\rangle \leq \varepsilon_{k+1} \enspace.
		\end{split}
	\end{equation}
\end{lemma}

\begin{proof}
    Let $(x, y) \in \X \times \Y$.
	Applying \Cref{proxnonexp} to the definition of $x_{k+1}$ yields
	\begin{equation}\label{eq:x_prox}
		\begin{split}
			& \frac{1}{2} \nor{x_{k + 1} - x}_T^2 - \frac{1}{2} \nor{x_k - x}_T^2 + \frac{1}{2} \nor{x_{k + 1} - x_k}_T^2 + \left[R(x_{k + 1})-R(x)\right] \\
			& \hspace{4cm} + \langle \tilde y_k, A\left(x_{k + 1} - x\right)\rangle + \langle \nabla F(x_k), x_{k + 1} - x\rangle \ \leq \ \varepsilon_{k+1} \enspace.
		\end{split}
	\end{equation}
	For the dual update, similarly,
	\begin{equation}\label{eq:y_prox}
		\begin{split}
			\frac{1}{2} \nor{y_{k + 1} - y}_{\Sigma}^2 -  \frac{1}{2} \nor{y_k - y}_{\Sigma}^2 + \frac{1}{2} \nor{y_{k + 1} - y_k}_{\Sigma}^2
			        + \langle y_{k + 1} - y, \bb - Ax_{k + 1}\rangle
			        \leq 0 \enspace.
		\end{split}
	\end{equation}
	Recall that $z := (x, y)$ and the definition of $V$.
	Sum \Cref{eq:x_prox,eq:y_prox}:
	\begin{equation}\label{eq:bound_prox}
		\begin{split}
			& V(z_{k + 1} - z) - V(z_k - z) +  V(z_{k + 1} - z_k) + \left[R\left(x_{k + 1}\right)-R(x)\right] \\
			&+ \langle \tilde y_k, A\left(x_{k + 1} - x\right)\rangle + \langle y_{k + 1} - y, \bb - Ax_{k + 1}\rangle + \langle \nabla F(x_k), x_{k + 1} - x\rangle\leq \varepsilon_{k+1} \enspace.
		\end{split}
	\end{equation}
	From the \Cref{desclemm},
	\begin{equation*}
		F(x_{k+1})\leq F(x_k)+ \langle \nabla F(x_k), x_{k + 1} - x_k\rangle+\frac{L}{2}\norm{x_{k+1}-x_k}{}^2 \enspace,
	\end{equation*}
	while from the convexity of $F$,
	\begin{equation*}
		F(x_k)+ \langle \nabla F(x_k), x - x_k\rangle\leq F(x) \enspace.
	\end{equation*}
	Summing the last two equations, one obtains the 3 points descent lemma:
	\begin{equation}\label{eq:three_pt}
		F(x_{k+1})\leq F(x)+ \langle \nabla F(x_k), x_{k + 1} - x\rangle+\frac{L}{2}\norm{x_{k+1}-x_k}{}^2 \enspace.
	\end{equation}
	Summing \Cref{eq:bound_prox,eq:three_pt},
	\begin{equation*}
		\begin{split}
			& V(z_{k + 1} - z) - V(z_k - z) + V(z_{k+1} - z_k)\\
			& + \left[R+F\right](x_{k+1})- \left[R+F\right](x) + \langle \tilde y_k, A\left(x_{k + 1} - x\right)\rangle + \langle y_{k + 1} - y, \bb - Ax_{k + 1}\rangle \\
			\leq \ \ & \frac{L}{2}\norm{x_{k+1}-x_k}{}^2 + \varepsilon_{k+1} \enspace.
		\end{split}
	\end{equation*}
	Now compute
	\begin{equation*}
		\begin{split}
			& \left[R+F\right](x_{k+1})- \left[R+F\right](x) + \langle \tilde y_k, A\left(x_{k + 1} - x\right)\rangle + \langle y_{k + 1} - y, \bb - Ax_{k + 1}\rangle\\
			= \ &   \left[\LLd(x_{k + 1}, y) - \LLd(x, y_{k + 1})\right]-\langle y,  Ax_{k + 1} - \bb \rangle  + \langle y_{k + 1}, Ax - \bb \rangle \\
			& + \langle \tilde y_k, A\left(x_{k + 1} - x\right)\rangle+  \langle  y_{k + 1} - y, \bb-Ax_{k + 1}\rangle\\
			= \ & \left[\LLd(x_{k + 1}, y) - \LLd(x, y_{k + 1})\right]- \langle y_{k + 1}-y, \bb \rangle - \langle y, Ax_{k + 1}\rangle + \langle y_{k + 1}, A x \rangle\\
			& + \langle \tilde y_k, A x_{k + 1}\rangle - \langle \tilde y_k, A x\rangle+  \langle y_{k + 1} - y, \bb\rangle - \langle y_{k + 1} - y, Ax_{k + 1}\rangle\\
			= \ & \left[\LLd(x_{k + 1}, y) - \LLd(x, y_{k + 1})\right]\\
			& - \langle y, Ax_{k + 1}\rangle+ \langle y_{k + 1}, A x\rangle + \langle \tilde y_k, Ax_{k + 1}\rangle - \langle \tilde y_k, A x\rangle -  \langle y_{k + 1}, Ax_{k + 1}\rangle+  \langle y, Ax_{k + 1}\rangle\\
			= \ &  \left[\LLd(x_{k + 1}, y) - \LLd(x, y_{k + 1})\right] + \langle y_{k + 1} - \tilde y_k, A\left(x -  x_{k + 1}\right)\rangle \enspace.
		\end{split}
	\end{equation*}
	Notice that
	\begin{equation}
		\frac{1}{2\tau_M}\norm{x_{k+1}-x_k}{}^2 \leq \frac{1}{2}\norm{x_{k+1}-x_k}{T}^2 \enspace.
	\end{equation}
	Finally,
	\begin{equation*}
		\begin{split}
			& V(z_{k + 1} - z) - V(z_k - z) + \frac{1-\tau_M L}{2\tau_M}\norm{x_{k+1}-x_k}{}^2+ \frac{1}{2}\norm{y_{k+1}-y_k}{\Sigma}^2\\
			& \hspace{3cm}
			+ \left[\LLd(x_{k + 1}, y) - \LLd(x, y_{k + 1})\right] + \langle y_{k + 1} - \tilde y_k, A\left(x -  x_{k + 1}\right)\rangle \leq \varepsilon_{k+1} \enspace.
		\end{split}
	\end{equation*}
\end{proof}
\begin{lemma}[First cumulating estimate]\label{lem:cum1}
	Let \Cref{ass:gen} hold.
    Let $(x_k,y_k)$ be the sequence generated by iterations \eqref{eq:algo} under \Cref{ass:TSigm}.
    Define $\omega:= 1 - \tau_M(L+\sigma_M\nor{A}^2)$.
    Then, for any $(x, y)\in\X \times\Y$ and for any $k\in\N$,
	\begin{equation}
		\begin{split}
			& \tfrac{1-\tau_M\sigma_M\norm{A}{}^2}{2\tau_M}\nor{x_{k} - x}^2
			+ \frac{1}{2}\norm{y_{k} - y}{\Sigma}^2
			+  \sum_{j=1}^{k}\left[\LLd(x_j,y) - \LLd(x, y_j)\right]
			+\frac{\omega}{2\tau_M}\sum_{j=1}^{k}\nor{x_j - x_{j - 1}}^2  \\
			&\leq    V(z_0 - z) + \sum_{j=1}^{k}\varepsilon_{j}.
		\end{split}
	\end{equation}
\end{lemma}
\begin{proof}
	We start from the inequality in \Cref{lem:onestep}, switching the index from $k$ to $j$. Recall that $\tilde y_j := 2 y_j - y_{j - 1}$, to get
	\begin{equation*}
		\begin{split}
			& V(z_{j + 1} - z) - V(z_j - z) + \frac{1-\tau_M L}{2\tau_M}\norm{x_{j+1}-x_j}{}^2+ \frac{1}{2}\norm{y_{j+1}-y_j}{\Sigma}^2 \\ & \hspace{6cm} +  \left[\LLd(x_{j+ 1}, y) - \LLd(x, y_{j + 1})\right] \\
			\leq \ \ & \varepsilon_{j+1}- \langle y_{j + 1} - \left(2y_j - y_{j - 1}\right), A\left(x - x_{j + 1}\right)\rangle \\
			= \ \ & \varepsilon_{j+1}- \langle y_{j+ 1} - y_j, A\left(x - x_{j + 1}\right)\rangle +\langle y_j-y_{j - 1}, A\left(x - x_{j + 1}\right)\rangle \\
			= \ \ & \varepsilon_{j+1}- \langle y_{j+ 1} - y_j, A\left(x - x_{j + 1}\right)\rangle +\langle y_j-y_{j - 1}, A\left(x - x_j\right)\rangle +\langle y_j-y_{j - 1}, A\left(x_j - x_{j + 1}\right)\rangle \enspace.
		\end{split}
	\end{equation*}
	Now focus on the term
	\begin{equation}
		\begin{split}
			\langle y_j-y_{j - 1}, A\left(x_j - x_{j + 1}\right)\rangle &=	\langle \Sigma^{\frac{1}{2}} \Sigma^{-\frac{1}{2}} \left(y_j-y_{j - 1}\right), A\left(x_j - x_{j + 1}\right)\rangle\\
			&=	\langle \Sigma^{-\frac{1}{2}}\left(y_j-y_{j - 1}\right), \Sigma^{\frac{1}{2}}A\left(x_j - x_{j + 1}\right)\rangle\\
			&\leq \nor{\Sigma^{-\frac{1}{2}}\left(y_j-y_{j - 1}\right)} \nor{\Sigma^{\frac{1}{2}}A\left(x_j - x_{j + 1}\right)}\\
			&\leq \frac{1}{2}\nor{\Sigma^{-\frac{1}{2}}\left(y_j-y_{j - 1}\right)}^2+\frac{1}{2} \nor{\Sigma^{\frac{1}{2}}A\left(x_j - x_{j + 1}\right)}^2\\
			&\leq \frac{1}{2}\norm{y_j-y_{j - 1}}{\Sigma}^2+\frac{\sigma_M\norm{A}{}^2}{2} \nor{ x_{j + 1}-x_j}^2 \enspace,
		\end{split}
	\end{equation}
	where we used Cauchy-Schwarz and Young inequalities.
	Then, using the definition of $\omega := 1 - \tau_M(L+\sigma_M\nor{A}^2)$, we have
	\begin{equation}\label{middleineq}
		\begin{split}
			& V(z_{j + 1} - z) - V(z_j - z) +  \left[\LL(x_{j+ 1}, y) - \LL(x, y_{j + 1})\right] \\
			& + \frac{\omega}{2\tau_M}\norm{x_{j+1}-x_j}{}^2+ \frac{1}{2}\norm{y_{j+1}-y_j}{\Sigma}^2 -\frac{1}{2} \norm{y_j-y_{j-1}}{\Sigma}^2\\
			\leq \ \ & \varepsilon_{j+1}- \langle y_{j + 1} - y_j, A\left(x - x_{j + 1}\right)\rangle +\langle y_j-y_{j- 1}, A\left(x - x_j\right)\rangle \enspace.
		\end{split}
	\end{equation}
	Imposing $y_{-1}=y_0$, summing-up \Cref{middleineq} from $j=0$ to $j=k-1$:
	\begin{equation*}
		\begin{split}
			& V(z_{k} - z) - V(z_0 - z) + \sum_{j=0}^{k-1}\left[\LLd(x_{j + 1}, y) - \LLd(x, y_{j + 1})\right] + \frac{\omega}{2\tau_M}\sum_{j=0}^{k-1}\nor{x_{j + 1} - x_j}^2  \\
			& \hspace{9cm}+ \frac{1}{2} \norm{y_k - y_{k - 1}}{\Sigma}^2 \\
    		& \leq  \sum_{j=0}^{k-1}\varepsilon_{j+1}- \langle y_{k} - y_{k-1}, A\left(x - x_{k}\right)\rangle\\
			& \leq  \frac{1}{2} \norm{y_{k} - y_{k-1}}{\Sigma}^2 +\frac{\sigma_M \nor{A}^2}{2} \nor{x_{k} - x}^2 +  \sum_{j=1}^{k}\varepsilon_{j} \enspace,
		\end{split}
	\end{equation*}
	where in the last inequality we used again Cauchy-Schwarz and Young inequalities as before. Reordering, we obtain the claim.
\end{proof}
\begin{lemma}[Second cumulative estimate]\label{lem:cum2}
	Let \Cref{ass:gen} hold.
	Let $(x_k,y_k)$ be the sequence generated by iterations \eqref{eq:algo} under \Cref{ass:TSigm}.
	Given $\xi>0$ and $\eta >0$, define $\theta := \xi - \tau_M(\xi L+\sigma_M \nor{A}^2)$ and $\rho:=\sigma_m(\eta-1)-\sigma_M\xi\eta$.
	Then, for any $z = (x, y)\in\X \times\Y$ and for any $k\in\N$,
	\begin{equation}
		\begin{split}
			& V(z_{k} - z)
			+ \frac{\theta}{2\tau_M\xi}\sum_{j=1}^{k}\nor{x_{j} - x_{j-1}}^2
			+\frac{\rho}{2\eta}\sum_{j=1}^{k}\nor{A x_{j}- Ax}^2
			+\sum_{j=1}^{k}\left[\LLd(x_{j }, y) - \LLd(x, y_{j})\right]\\
			\leq \ &  V(z_0 - z) + \sum_{j=1}^{k} \varepsilon_j + \frac{\sigma_m \left(\eta-1\right)k}{2}\nor{Ax-\bb}^2 \enspace.
		\end{split}
	\end{equation}
\end{lemma}

\begin{proof}
	In a similar fashion as in the previous proof, we start again from the main inequality in \Cref{lem:onestep}, switching the index from $k$ to $j$.
	Since $\tilde y_j = y_j + (y_j - y_{j - 1}) = y_j + \Sigma (A x_j - b^\delta)$ and $y_{j+1} - y_j = \Sigma (A x_{j+1} - b^\delta)$, we get
	\begin{equation*}
		\begin{split}
			& V(z_{j + 1} -z) - V(z_j - z) + \frac{1-\tau_M L}{2 \tau_M} \nor{x_{j+1} - x_j}^2 + \frac{1}{2}\norm{\Sigma \left(A x_{j+1} - \bb\right)}{\Sigma}^2 \\
			& \hspace{10cm} +  \left[\LLd(x_{j+ 1}, x) - \LLd(x, y_{j + 1})\right]\\
			\leq \ & \varepsilon_{j+1}+\langle  y_{j+1} - y_j - \Sigma \left(A x_{j} - \bb\right) , A x_{j+1} - A x \rangle\\
			= \ & \varepsilon_{j+1}+\langle \Sigma A\left( x_{j+1} - x_j\right), A x_{j+1} - Ax\rangle \enspace.
		\end{split}
	\end{equation*}
	Now estimate
	\begin{equation*}
		\begin{split}
			\frac{1}{2}\norm{\Sigma \left(A x_{j+1} - \bb\right)}{\Sigma}^2
			& =\frac{1}{2} \langle \Sigma \left(A x_{j+1} - \bb\right), A x_{j+1} - \bb\rangle\\
			& \geq \frac{\sigma_m}{2} \norm{A x_{j+1} - \bb}{}^2\\
			& = \frac{\sigma_m}{2}\nor{A x_{j+1} - Ax}^2 + \frac{\sigma_m}{2}\nor{Ax-\bb}^2  + \sigma_m \langle A x_{j+1} - Ax, Ax-\bb\rangle \enspace .
		\end{split}
	\end{equation*}
	So,
	\begin{equation*}
		\begin{split}
			& V(z_{j + 1} - z) - V(z_j - z) + \frac{1-\tau_M L}{2\tau_M} \nor{x_{j+1} - x_j}^2 + \frac{\sigma_m}{2}\nor{A x_{j+1} - Ax}^2  \\
			& \hspace{10cm} + \left[\LLd(x_{j+1},y) - \LLd (x,y_{j+1})\right]\\
			\leq \ & \varepsilon_{j+1}+\langle \Sigma A\left( x_{j+1} - x_j\right), A x_{j+1} - Ax\rangle + \sigma_m \langle A x_{j+1} - Ax, \bb-Ax \rangle - \frac{\sigma_m}{2}\nor{Ax-\bb}^2\\
			\leq \ & \varepsilon_{j+1}+ \frac{\sigma_M\nor{A}^2}{2\xi}\nor{x_{j+1} - x_j}^2 +\frac{\xi\sigma_M}{2}\nor{A x_{j+1}- Ax}^2 - \frac{\sigma_m}{2}\nor{Ax - b^\delta}^2\\
			& +\frac{\sigma_m}{2\eta}\nor{A x_{j+1}- Ax}^2 + \frac{\sigma_m \eta}{2}\nor{Ax-\bb}^2 \enspace.
		\end{split}
	\end{equation*}
	In the last inequality we used three times Cauchy-Schwarz inequality and twice Young inequality with parameters $\xi>0$ and $\eta >0$.
	Then, reordering and recalling the definitions of $\theta := \xi - \tau_M(\xi L+\sigma_M \nor{A}^2)$,  we obtain
	\begin{equation*}
		\begin{split}
			& V(z_{j + 1} - z) - V(z_j - z)  + \frac{\theta}{2\tau_M \xi}\nor{x_{j+1} - x_j}^2 + \frac{\sigma_m(\eta-1)-\sigma_M\xi\eta}{2\eta}\nor{A x_{j+1}- Ax}^2 \\
		+  & \left[\LLd(x_{j+1},y) - \LLd\left(x,y_{j+1}\right)\right]	\quad \leq \quad \varepsilon_{j+1}+\frac{\sigma_m \left(\eta-1\right)}{2}\nor{Ax-\bb}^2 \enspace.
		\end{split}
	\end{equation*}
	Summing-up the latter from $j=0$ to $j=k-1$, we get
	\begin{equation*}
		\begin{split}
			& V(z_{k} - z) - V(z_0 - z) +\frac{\theta}{2\tau_M\xi}\sum_{j=0}^{k-1}\nor{x_{j + 1} - x_j}^2 +\frac{\sigma_m(\eta-1)-\sigma_M\xi\eta}{2\eta}\sum_{j=0}^{k-1}\nor{A x_{j+1}- Ax}^2\\
			& + \sum_{j=0}^{k-1}\left[\LLd(x_{j + 1}, y) - \LLd(x, y_{j + 1})\right] \ \ \leq \ \ \sum_{j=0}^{k-1}\varepsilon_{j+1}+\frac{\sigma_m \left(\eta-1\right)k}{2}\nor{Ax-\bb}^2 \enspace.
		\end{split}
	\end{equation*}
	By trivial manipulations, we get the claim.
\end{proof}

\section{Proofs of main results}

\subsection{Proof of \Cref{prop:convergence_cp}}
\label{app:condat_pf}

\condatconvergence*
\begin{proof}
Up to a change of initialization and offset of index, the steps of algorithm \eqref{eq:algo} when $\varepsilon_k = 0$ correspond to
\begin{align}\label{orginal}
    \begin{cases}
	y_{k + 1} = y_k + \Sigma \left(Ax_{k} - \bbs\right)\\
    x_{k+1} = \prox^{T}_R(x_k - \T\nabla F(x_k) - \T A^*(2 y_{k+1} - y_{k})) \enspace.
	\end{cases}
\end{align}
We now show that the previous iterations correspond to Algorithm 3.2 in \cite{Condat13}, setting $\sigma=\tau=1$ and applying it in the metrics defined by the preconditioning operators; namely, in the primal and dual spaces $(\X, \ \langle T^{-1} \cdot,  \cdot \rangle)$ and $(\Y, \ \langle \Sigma\cdot,  \cdot \rangle)$ - respectively.
Comparing problem \eqref{primal} with (1) in \cite{Condat13}, their notation in our setting reads as $F=F,\ G=R,\ H=\iota_{\left\{\bbs\right\}}$ and $K=A$.
The Fenchel conjugate of $H$ in $(\Y, \ \langle \Sigma\cdot,  \cdot \rangle)$ is
\begin{equation}
    \begin{split}
        H^{\star}(y)&=\sup_{z\in\Y} \left\{\langle \Sigma z, y \rangle-\iota_{\left\{\bbs\right\}}(z)\right\}=\langle \Sigma \bbs, y \rangle
    \end{split}
\end{equation}
and its proximal-point operator, again in $(\Y, \ \langle \Sigma\cdot,  \cdot \rangle)$, is
\begin{equation}
    \begin{split}
        \prox_{H^{\star}}(y)&=\argmin_{z\in\Y} \left\{\langle \Sigma \bbs, z \rangle+\frac{1}{2}\langle \Sigma(z-y),z-y\rangle \right\}=y-\bbs \enspace.
    \end{split}
\end{equation}
The gradient of $F$ in $(\X, \ \langle \T^{-1}\cdot,  \cdot \rangle)$ is denoted by $\nabla_{T} F(x)$ and satisfies, for $x$ and $v$ in $\X$,
\begin{equation*}
    \langle T^{-1}\nabla_{T} F(x), v \rangle=\langle \nabla F(x), v \rangle \enspace .
\end{equation*}
It is easy to see that one has $\nabla_T F(x)=T \nabla F(x)$.\\
The adjoint operator of $K: \ (\X, \ \langle \T^{-1}\cdot,  \cdot \rangle)\to (\Y, \ \langle \Sigma\cdot,  \cdot \rangle)$ satisfies, for every $(x,y)\in\X \times \Y$,
\begin{equation}
    \begin{split}
        \langle T^{-1}K^*y,x \rangle&=\langle \Sigma Kx, y \rangle=\langle \Sigma Ax, y \rangle = \langle x, A^*\Sigma y \rangle \enspace,
    \end{split}
\end{equation}
implying that $T^{-1}K^*=A^*\Sigma$ and so that $K^*=TA^*\Sigma$.
Then Algorithm 3.2 in \cite{Condat13} (with $\sigma=\tau=1$, $\rho_k=1$ for every $k\in\N$ and no errors involved) is:
\begin{align*}
    \begin{cases}
	\bar{y}_{k + 1} = \prox_{H^{\star}}(\bar{y}_k+K\bar{x}_k)\\
    \bar{x}_{k+1} = \prox_R(\bar{x}_k - \nabla_T F(\bar{x}_k) - K^*(2 \bar{y}_{k+1} - \bar{y}_{k})) \enspace,
	\end{cases}
\end{align*}
and becomes, applied to our setting in the spaces $(\X, \ \langle T^{-1} \cdot,  \cdot \rangle)$ and $(\Y, \ \langle \Sigma\cdot,  \cdot \rangle)$,
\begin{align*}
    \begin{cases}
	\bar{y}_{k + 1} =\bar{y}_k + A\bar{x}_k-\bbs\\
    \bar{x}_{k+1} = \argmin_{x\in\X} \left\{R(x)+\frac{1}{2}\norm{x-\left[\bar{x}_k - T\nabla F(\bar{x}_k) - TA^*\Sigma(2 \bar{y}_{k+1} - \bar{y}_{k})\right]}{T^{-1}}^2\right\} \enspace.
	\end{cases}
\end{align*}
Define the variable $\bar{z}_k=\Sigma \bar{y}_k$ and multiply the first line by $\Sigma$. Then,
\begin{align*}
    \begin{cases}
	\bar{z}_{k + 1} =\bar{z}_k + \Sigma\left(A\bar{x}_k-\bbs\right)\\
    \bar{x}_{k+1} = \prox_R^T\left(\bar{x}_k - T\nabla F(\bar{x}_k) - TA^*(2 \bar{z}_{k+1} - \bar{z}_{k})\right) \enspace.
	\end{cases}
\end{align*}
Comparing the previous with \eqref{orginal}, we get that they are indeed the same algorithm.
To conclude, we want to use Theorem 3.1 in \cite{Condat13}, that ensures the weak convergence of the sequence generated by the algorithm to a saddle-point.
It remains to check that, under our assumptions, the hypothesis of the above result are indeed satisfied; namely, that
\begin{equation}\label{cond}
    1-\norm{K}{}^2-\frac{L_T}{2} \geq 0 \enspace,
\end{equation}
where $\norm{K}{}$ represents the operator norm of $K: \ (\X, \ \langle \T^{-1}\cdot,  \cdot \rangle)\to (\Y, \ \langle \Sigma\cdot,  \cdot \rangle)$ and $L_T$ is the Lipschitz constant of $\nabla_T F$.
Notice that
$$\norm{K}{}^2=\sup_{x\in\X} \frac{\langle \Sigma Ax,Ax\rangle }{\langle T^{-1}x,x\rangle}\leq \sigma_M\tau_M\norm{A}{}^2 \enspace.$$
Moreover, $L_T\leq \tau_M L$.
Indeed, for every $x$ and $x'\in\X$,
$$\norm{\nabla_T F(x')-\nabla_T F(x)}{}=\norm{T \nabla F(x')-T \nabla F(x)}{}\leq \tau_M \norm{ \nabla F(x')- \nabla F(x)}{} \enspace.$$
Then, by \Cref{ass:omega} and the previous considerations,
$$0\leq 1-\tau_M(L + \sigma_M\norm{A}{}^2)\leq 1-L_T-\norm{K}{}^2\leq 1-\frac{L_T}{2}-\norm{K}{}^2 \enspace.$$
In particular, \eqref{cond} is satisfied and we the claim is proved.
\end{proof}
\subsection{Proof of \Cref{prop:feas_lagrangian}}
\label{app:pf_feaslagrangian}

\feaslagrangian*

\begin{proof} %
For simplicity, denote $J:=R+F$. First notice that, for our problem, the Lagrangian gap is equal to the Bregman divergence.
	Indeed, using $-A^* \ys \in \partial J(\xs)$ and $A \xs = \bbs$:
	\begin{align}\label{eq:breg_is_gap}
    	\LLs(x, \ys) - \LLs(\xs, y)
    	&= J(x) - J(\xs) + \langle \ys, A x - \bbs \rangle - \langle y, A  \xs - \bbs \rangle  \nonumber \\
    	&= J(x) - J(\xs) + \langle A^* \ys, x - \xs \rangle = D_J^{-A^* \ys}(x, \xs)\enspace,
	\end{align}
	We then show that if $v \in \partial J(\xs)$ and $D_{J}^{v}(x, \xs) = 0$, then $v \in \partial J(x)$.
	Indeed, $J(x) - J(\xs) - \langle v, x - \xs \rangle = 0$ and so, for all $x' \in \X$,
	\begin{equation}
	J(x') \geq J(\xs) + \langle v, x' - \xs \rangle = J(x) - \langle v, x - \xs \rangle + \langle v, x' - \xs \rangle = J(x) + \langle v, x' - x \rangle \enspace .
	\end{equation}
	\Cref{prop:feas_lagrangian} follows by taking $v = -A^* \bbs$.
\end{proof}

\subsection{Proof of \Cref{prop:early_stop}}\label{app:pf_earlystop}
\earlystop*
\begin{proof}
    Recall that we denote $z = (x, y) \in \X \times \Y$ a primal-dual pair, and define
    \begin{equation}
    	V(z) := \frac{1}{2}\norm{x}{T}^2+\frac{1}{2}\norm{y}{\Sigma}^2 \enspace.
    \end{equation}
	Use \Cref{lem:cum1} at $x=\xs$ and $y=\ys$, to get
		\begin{equation}\label{main}
		\begin{split}
			& \tfrac{1-\tau_M\sigma_M\norm{A}{}^2}{2\tau_M}\nor{x_{k}
			  - \xs}^2
			  + \tfrac{1}{2}\norm{y_{k} - \ys}{\Sigma}^2
			  +  \sum_{j=1}^{k}  [\LL^\delta(x_j, \ys) - \LL^\delta(\xs, y_j)]
			+\tfrac{\omega}{2\tau_M}\sum_{j=1}^{k}\nor{x_j - x_{j - 1}}^2  \\
			& \leq  V(z_0 - \zs)+\sum_{j=1}^{k}\varepsilon_{j} \enspace.
		\end{split}
	\end{equation}
Notice that
	\begin{equation}\label{notice}
		\LL^\delta(x_{j}, \ys) - \LL^\delta(\xs, y_{j})
		    = \LLs(x_{j}, \ys) - \LLs(\xs, y_{j}) + \langle y_{j} - \ys, \bb-\bbs \rangle \enspace .
	\end{equation}
	Then,
	\begin{equation}\label{here}
		\begin{split}
			& \tfrac{1-\tau_M\sigma_M\norm{A}{}^2}{2\tau_M}\nor{x_{k} - \xs}^2 + \frac{1}{2}\norm{y_{k} - \ys}{\Sigma}^2 +  \sum_{j=1}^{k}\left[\LLs(x_j,\ys) - \LLs(\xs, y_j)\right]
			+\tfrac{\omega}{2\tau_M}\sum_{j=1}^{k}\nor{x_j - x_{j - 1}}^2
			\\& \leq   V(z_0 - \zs)+\sum_{j=1}^{k}\varepsilon_{j}+ \delta \sum_{j=1}^{k} \nor{y_{j} - \ys}.
		\end{split}
	\end{equation}
	Recall that $\LLs(x,\ys) - \LLs(\xs, y)\geq 0 $ for every $\left(x,y\right)\in\X\times\Y$.
	Moreover, $\omega\geq 0$ by \Cref{ass:omega} and so  $1-\tau_M\sigma_M\norm{A}{}^2 \geq 0$.
	Then, for every $j\in\N$, we have that
	\begin{equation}
		\norm{y_j - \ys}{\Sigma}^2  \ \leq \  2 V(z_0 - \zs) + 2\sum_{i=1}^{j}\varepsilon_{i} + 2\delta \sum_{i=1}^{j}\nor{y_i - \ys}
	\end{equation}
	and so
	\begin{equation}\label{eq:bound_norm_y}
		\norm{y_j - \ys}{}^2  \ \leq \  2\sigma_M \left[V(z_0 - \zs) +\sum_{i=1}^{j}\varepsilon_{i}\right] + 2\delta\sigma_M \sum_{i=1}^{j}\nor{y_i - \ys} \enspace.
	\end{equation}
	Apply \Cref{lem:u_n_upper_bound} to \Cref{eq:bound_norm_y} with $u_j=\nor{y_j - \ys}$, $S_j=2 \sigma_M \left[V(z_0 - \zs)+\sum_{i=1}^{j}\varepsilon_{i} \right]$ and $\lambda=2\delta\sigma_M$.
	We get, for $1\leq j\leq k$,
	\begin{equation}\label{questa}
    \begin{split}
		\nor{y_j - \ys}   & \leq \delta \sigma_M j + \sqrt{2\sigma_M \left[V(z_0 - \zs)+\sum_{i=1}^{j}\varepsilon_{i} \right]+\left(\delta\sigma_M j\right)^2}\\
		&\leq 2\delta \sigma_M k + \sqrt{2\sigma_M \left[V(z_0 - \zs)+\sum_{i=1}^{k}\varepsilon_{i} \right]} \enspace.
	\end{split}
	\end{equation}
	Insert the latter in \Cref{here}, to obtain
	\begin{equation*}
		\begin{split}
			&\sum_{j=1}^{k}\left[\LLs(x_j,\ys) - \LLs(\xs, y_j)\right]  \\
			& \quad \quad \leq  V(z_0 - \zs)+\sum_{j=1}^{k}\varepsilon_{j} + \delta \sum_{j=1}^{k} \left(2\delta\sigma_M k + \sqrt{2\sigma_M \left[V(z_0 - \zs)+\sum_{i=1}^{k}\varepsilon_{i} \right]}\right)\\
			& \quad \quad =  V(z_0 - \zs)  + \sum_{j=1}^{k}\varepsilon_{j}+\delta k \sqrt{2\sigma_M \left[V(z_0 - \zs)+\sum_{i=1}^{k}\varepsilon_{i} \right]} + 2 \delta^2 \sigma_M k^2 \\
			& \quad \quad \leq  V(z_0 - \zs) + C_0 k \delta + \delta k \left( \sqrt{2\sigma_M V(z_0 - \zs)} + \sqrt{2\sigma_M C_0 k \delta} \right)  + 2 \delta^2 \sigma_M k^2 \enspace,
		\end{split}
	\end{equation*}
	where the last line uses $\sqrt{a + b} \leq \sqrt{a} + \sqrt{b}$.
	By Jensen's inequality, we get the first claim.\\
	For the second result, apply \Cref{lem:cum2} at $x=\xs$ and $y=\ys$:
		\begin{equation}
		\begin{split}
			& V(z_{k} - \zs)
			+ \frac{\theta}{2\tau_M\xi} \sum_{j=1}^{k}\nor{x_{j} - x_{j-1}}^2
			+ \frac{\rho}{2\eta}\sum_{j=1}^{k}\nor{A x_{j} - A\xs}^2µ
			+ \sum_{j=1}^{k}\left[\LL^\delta(x_{j }, \ys) - \LL^\delta(\xs, y_{j})\right]\\
			& \leq  V(z_0 - \zs)+\sum_{j=1}^{k} \varepsilon_j + \frac{\sigma_m \left(\eta-1\right)k}{2}\nor{A\xs-\bb}^2 \enspace.
		\end{split}
	\end{equation}
	Using \Cref{notice,questa}, we have
	\begin{equation}
		\begin{split}
			& V(z_{k} - \zs) +\frac{\theta}{2\tau_M\xi}\sum_{j=1}^{k}\nor{x_{j} - x_{j-1}}^2 +\frac{\rho}{2\eta}\sum_{j=1}^{k}\nor{A x_{j}- \bbs}^2+\sum_{j=1}^{k}\left[\LLs(x_{j }, \ys) - \LLs(\xs, y_{j})\right]\\
			\leq \ & V(z_0 - \zs)+\sum_{j=1}^{k} \varepsilon_j +\sum_{j=1}^{k}\langle y_j-\ys,\bbs-\bb \rangle+ \frac{\sigma_m \left(\eta-1\right)k}{2}\nor{\bbs-\bb}^2\\
			\leq \ & V(z_0 - \zs)+\sum_{j=1}^{k} \varepsilon_j +\delta\sum_{j=1}^{k}\nor{ y_j-\ys}+ \frac{\sigma_m \left(\eta-1\right)k}{2}\delta^2\\
			\leq \ & V(z_0 - \zs)+\sum_{j=1}^{k} \varepsilon_j+2\sigma_M\delta^2 k^2 + \delta k \sqrt{2\sigma_M \left[V(z_0 - \zs)+\sum_{i=1}^{k}\varepsilon_{i} \right]}+\frac{\sigma_m \left(\eta-1\right)k}{2}\delta^2 \enspace.
		\end{split}
	\end{equation}
Recall that $\theta\geq 0$ and that $\LLs(x,\ys) - \LLs(\xs, y)\geq 0 $ for every $\left(x,y\right)\in\X\times\Y$.
By Jensen's inequality, rearranging the terms and using $\sum_{i=1}^k \varepsilon_i \leq C_0 k \delta $, we get the claim.
The exact values of the constants of \Cref{prop:early_stop} are therefore:
\begin{equation}\label{constants}
\begin{split}
    C_1 &= V(z_0 - \zs)  \enspace, \\
    C_2 &= C_0 + \sqrt{2\sigma_M V(z_0 - \zs)} \enspace, \\
    C_3 &= \sqrt{2 \sigma_M C_0} \enspace, \\
    C_4 &= 2 \sigma_M \enspace, \\
    C_5 &= \frac{2\eta}{\rho} C_1 \enspace, \\
    C_6 &= \frac{2\eta}{\rho} C_2 \enspace, \\
    C_7 &= \frac{2\eta}{\rho} C_3 \enspace, \\
    C_8 &= \frac{2\eta}{\rho} C_4 \enspace, \\
    C_9 &= \frac{\eta \sigma_m(\eta - 1)}{\rho} \enspace, \\
\end{split}
\end{equation}
\
\end{proof}

\subsection{Example of divergence in absence of noisy solution (see \Cref{sub:divergence})}\label{app:divergence}

We present an example in which the primal exact problem has solution, but the noisy one does not and the averaged primal iterates generated by Algorithm \eqref{eq:algo} indeed diverge.
First note that, if the function $R$ has bounded domain, the primal iterates remain bounded. So, to exhibit a case of divergence of the primal iterates, we consider a function $R$ with full domain: set $R (\cdot)= \frac{1}{2} \nor{\cdot}^2$ (and $F=0$). The exact problem is then
\begin{equation}\label{exact_sample}
    \min_{x \in \mathcal{X}}\  \frac{1}{2} \nor{x}^2   \quad \text{s.t.} \quad Ax = b^{\star} \enspace.
\end{equation}
 Now consider a noisy datum $b^{\delta}$ such that $Ax=b^\delta$ does not have a solution. If the associated normal equation, namely $A^*Ax=A^*b^{\delta}$ is feasible, in \Cref{sec:unfeas} we prove not only boundedness of the iterates but also convergence to a normal solution. On the contrary,  to get divergence of the iterates, here we consider a classic scenario in which the perturbation of the exact data generates an unfeasible constraint, even for the associated normal equation. We recall that this may happen only in the infinite dimensional setting, as when $R(A)$ is finite dimensional, it is also closed and a solution to the normal equation always exists. %
As a prototype of ill-posed problem, let  $\mathcal{X}=\mathcal{Y}=\ell^2$ and $A$ be defined by, for every $x\in\ell^2$ and for every $i\in\N$,
\begin{equation*}
    (Ax)^i = a^ix^i \enspace,
\end{equation*}
where, for every $i\in\N$, $a^i\in (0,M)$ for a fixed constant $M>0$ and $\inf_{i\in\N} a_i =0$. Note that $A\colon \ell^2\to\ell^2$ is well-defined, linear, continuous, self-adjoint and compact.
Let $b^\star$ in the range of $A$ and denote by $x^\star$ the unique solution to $\mathcal{P}^{\star}$ defined in \Cref{exact_sample}; namely, $(x^\star)^i:=
    (b^\star)^i/a^i$ for every $i\in\N$. In particular, the $(b^\star)^i$ are such that $x^\star$ belongs to $\ell^2$. Let also $b^{\delta}\in\ell^2$ with $\|b^\delta-b^\star\|\leq \delta$, but such that the noisy equation does not have a normal solution. Defining, for every $i\in\N$,
\begin{equation}\label{xbar}
(x^\delta)^i:=
    (b^\delta)^i/a^i,
\end{equation}
the previous means that $x^\delta$ does not belong to $\ell^2$. For an explicit example, consider $a^i=1/i, (b^\star)^i=1/i^2$ and $(b^\delta)^i=(b^\star)^i+C/i$, with $C={\delta/}{ \sqrt{\sum_{j=1}^{+\infty} 1/j^2}}$.\\
Apply the algorithm with step-sizes $\sigma>0$ and $\tau>0$ such that $\sigma \tau < 1/\norm{A}{}^2$ and notice that it implies, for every $i\in\N$, $\sigma \tau < 1/(a^i)^2$.
As $a^i>0$ for every $i\in\N$, the coordinates of the averaged sequence $(\hat{x}_k^i)$ are convergent to a solution of the following (one-dimensional) optimization problem:
\begin{align*}
    \mathcal{P}^i & := \argmin_{x^i\in\R} \left\{ \frac{1}{2}(x^i)^2: \ \ a^ix^i=(b^\delta)^i\right\} = \left\{ \frac{(b^\delta)^i}{a^i} \right\}. %
\end{align*}
Hence, for the primal-dual algorithm, if $x^{\delta}\notin\ell^2$, then $(\hat{x}_k)$ diverges. Indeed, by contradiction, suppose that $(\hat{x}_k)$ is bounded. As it is bounded and converges coordinate-wise to $x^{\delta}$, then it weakly converges to $x^{\delta}$. But this is not possible since $x^{\delta}$ is not in $\ell^2$.\\
Note that the problem considered in this example can be treated by Landweber method and it is well-known that also the iterates generated by this method, while being different from the ones of primal-dual algorithm, diverge.

\section{Sparse recovery}
\subsection{Proof of \cref{prop:extended_supp}}\label{app:sub:sparse_proof}

\extendedsupp*
\begin{proof}
Recall that $\xs, \ys$ is a primal-dual solution, hence $-A^* \ys \in \partial \nor{\xs}_1$.
For every $i \in \N$ we have that $\left[\partial \norm{\cdot}{1}\right]_i(\xs)  \subseteq \left[-1,1\right]$ and so  $\abs{\left(A^*\ys\right)_i}\leq 1$.
Recall that $\Gamma_C:=\N\setminus \Gamma$.
As $A^*\ys$ belongs to $\X=\ell^2(\N; \R)$, we have
\begin{equation}\label{eq:finite_supp}
	\sum_{i\in\N}^{} \abs{\left(A^*\ys\right)_i}^2<+\infty.
\end{equation}
Indeed, $m\leq 1$ by definition and from \Cref{eq:finite_supp} the coefficients $\abs{\left(A^*\ys\right)_i}$ converge to $0$ (and so they can not accumulate at $1$).
We have also that
\begin{equation*}
	\begin{split}
		D^{-A^*\ys}(x,\xs) & =\sum_{i\in\N}\left[\abs{x_i}-\abs{\xs_i}+\left(A^*\ys\right)_i\left(x_i-\xs_i\right)\right]\\ & =\sum_{i\in\N}\left[\abs{x_i}+\left(A^*\ys\right)_i x_i\right]\\
		& \geq \sum_{i\in\Gamma}\left[\abs{x_i}-\underbrace{\abs{\left(A^*\ys\right)_i}}_{=1}\abs{x_i}\right]+\sum_{i\in\Gamma_C}\left[\abs{x_i}-\underbrace{\abs{\left(A^*\ys\right)_i}}_{\leq m} \abs{x_i} \right]\\
		&\geq (1-m) \sum_{i\in\Gamma_C}^{} \abs{x_i}.
	\end{split}
\end{equation*}
\end{proof}

\subsection{Tykhonov regularization: Lasso}
\label{app:sub:tycho_sparse}
For Tykhonov regularisation, the results in terms of Bregman divergence and feasibility are the following.
\begin{lemma}[\cite{grasmair2011necessary}, Lemma 3.5]\label{eq:lasso}
	Let $A\xs=\bbs$, $-A^*\ys \in \partial \nor{\cdot}_1(\xs)$ and, for $\alpha>0$,
	\begin{equation}\label{TykProbl}
		x_\alpha \in \argmin_{x\in\X} \left\{\nor{Ax - b^\delta}^2 + \alpha \nor{x}_1 \right\}.
	\end{equation}
Then it holds that
\begin{equation*}
  \norm{Ax_{\alpha}-\bbs}{} \leq \delta +\alpha \norm{\ys}{}
\quad \quad \text{and} \quad \quad
  D^{-A^*\ys}(x_{\alpha},\xs)\leq \frac{\left(\delta+\alpha\norm{\ys}{}/2\right)^2}{\alpha}.
\end{equation*}
\end{lemma}
The previous bounds, combined with \Cref{ass:CS} and the last inequality in \Cref{lem:CS}, lead naturally to the following corollary.
\begin{corollary}[\cite{grasmair2011necessary}, Theorem 5.6]\label{cor:Tyk}
Suppose \Cref{ass:CS} holds. Then, for $x_\alpha$ defined as in \Cref{eq:lasso} and $C:=\alpha/\delta$,
\begin{align*}
  & \norm{Ax_\alpha-\bbs}{} \leq \left(1+CW_s\right) \delta
\quad \quad \text{and} \\
 & \norm{x_{\alpha}-\xs}{}
     \leq Q_s \left(1+CW_s\right) \delta+ \frac{1+Q_s\norm{A}{}}{1-M_s} \frac{\left(1+CW_s/2\right)^2}{C} \delta.
\end{align*}
\end{corollary}

\section{Proofs of \Cref{sec:unfeas}}\label{app:unfeas}

\subsection{Proof of \Cref{lem:unfeas}}\label{app:unfeas_1}
\unfeas*
\begin{proof}
From \Cref{lem:cum1}, for any $(x, y)\in\X \times\Y$ and for any $k\in\N$, we have
	\begin{equation}
		\begin{split}
			& \frac{1-\tau_M\sigma_M\norm{A}{}^2}{2\tau_M}\nor{x_{k} - x}^2 + \frac{1}{2\sigma}\norm{y_{k} - y}{\Sigma}^2 +  \sum_{j=1}^{k}\left[\LL(x_j,y) - \LL(x, y_j)\right]
			+\frac{\omega}{2\tau_M}\sum_{j=1}^{k}\nor{x_j - x_{j - 1}}^2  \\ &\leq   V(z_0 - z)+\sum_{j=1}^{k}\varepsilon_{j},
		\end{split}
	\end{equation}
	where $\omega:= 1 - \tau_M(L+\sigma_M\nor{A}^2) \geq 0$ by \Cref{ass:omega}.
	Using Jensen's inequality, we get
	\begin{equation}
		\begin{split}
			& \LL(\hat{x}_k,y) - \LL(x, \hat{y}_k)  \leq  \frac{1}{k} \left[V(z_0 - z)+\sum_{j=1}^{+\infty}\varepsilon_{j}\right].
		\end{split}
	\end{equation}
	Let $\left(x_{\infty},y_{\infty}\right)$ be a weak cluster point of $(\hat{x}_k,\hat{y}_k)$; namely, there exists a subsequence $(\hat{x}_{k_j},\hat{y}_{k_j})\subseteq (\hat{x}_k,\hat{y}_k)$ such that
	$(\hat{x}_{k_j},\hat{y}_{k_j})\rightharpoonup(x_{\infty},y_{\infty})$.
	By weak lower-semicontinuity of $R$ and $F$, for every $(x, y)\in\X \times\Y$,
	\begin{equation}
		\begin{split}
			& \LL(x_{\infty},y) - \LL(x, y_\infty) \leq \liminf_{j} \LL(\hat{x}_{k_j},y) - \LL(x, \hat{y}_{k_j}) \leq \liminf_{j} \frac{1}{k_j} \left[V(z_0 - z)+\sum_{j=1}^{+\infty}\varepsilon_{j}\right]=0.
		\end{split}
	\end{equation}
	Thus $(x_{\infty},y_{\infty})$ is a saddle-point for the Lagrangian.\\
	Now suppose that the set of saddle-points of $\mathcal{L}$ is empty. Assume also, for contradiction, that $(\hat{x}_k,\hat{y}_k)$ does not diverge.
	Then we can extract a bounded subsequence, that consequently admits a weakly converging subsequence.
	But then, the limit is a saddle-point, which contradicts the assumption.
\end{proof}

\subsection{Proof of \Cref{lem:unfeas_sameit}}\label{app:unfeas_2}
\unfeassameit*

\begin{proof}

As $\tilde{\mathcal{C}}\neq \emptyset$, there exists $x^b\in\X$ such that $A^*Ax^b=A^*b$.
First consider the algorithm in \eqref{algo:easy}.
Note that, for every $k\in\N$, $\tilde{y}_k=y_k+\sigma\left(Ax_k-b\right)$ and multiply the last step by $A^*$.
We get, for every $k\in\mathbb{N}$,
\begin{equation*}
    \begin{split}
        x_{k+1} &=\prox^{}_{\tau R} (x_k- \tau \nabla F(x_k)-\tau A^*y_{k}-\sigma\tau A^*A(x_k-x^b))\\
        A^*y_{k+1}&=A^*y_{k}+\sigma A^*A(x_{k+1}-x^b).
    \end{split}
\end{equation*}
Recall that $S:=(A^*A)^{\frac{1}{2}}$ and introduce $p_k:=A^*y_k$.
Then the primal sequence $\seq{x_k}$ is equivalently defined by the following recursion: given $x_0$ and $p_{0}=A^*y_{0}$, for every $k\in\mathbb{N}$,
\begin{equation}\label{compare}
    \begin{split}
        x_{k+1} &=\prox^{}_{\tau R} (x_k- \tau \nabla F(x_k)-\tau p_{k}-\sigma\tau S^2(x_k-x^b))\\
        p_{k+1}&=p_{k}+\sigma S^2(x_{k+1}-x^b).
    \end{split}
\end{equation}
As $A^*y_{-1}$ belongs to $R(A^*)$ and $R(A^*)=R(S)$ \cite[Prop 2.18]{engl1996regularization}, there exists $v_{-1}$ such that $Sv_{-1}=A^*y_{-1}$.
Now consider the primal-dual algorithm applied to problem \eqref{leastsquares} starting at $u_0=x_0$, $v_{-1}$ and $v_0=v_{-1}+\sigma(Su_0-Sx^b)$.
It reads as: for every $k\in\N$,
\begin{equation*}
    \begin{split}
        \tilde{v}_k&=2v_k-v_{k-1}\\
        u_{k+1}&=\prox_{\tau R} (u_k-\tau \nabla F(u_k)-\tau S\tilde{v}_k)\\
        v_{k+1}&=v_k+\sigma(Su_{k+1}-Sx^b).
    \end{split}
\end{equation*}
Then, noticing that $\tilde{v}_k=v_k+\sigma\left(Su_k-Sx^b\right)$ and multiplying the last step by $S$,
\begin{equation*}
    \begin{split}
        u_{k+1}&=\prox_{\tau R} (u_k-\tau \nabla F(u_k)-\tau Sv_{k}-\sigma\tau S^2(u_k-x^b))\\
        Sv_{k+1}&=Sv_{k}+\sigma S^2(u_{k+1}-x^b) \enspace.
    \end{split}
\end{equation*}
Define the change of variable $q_k:=Sv_k$, so that $q_{-1}=Sv_{-1}=A^*y_{-1}$ and
$$q_0=Sv_0=S\left(v_{-1}+\sigma(Su_0-Sx^b)\right)=A^*(y_{-1}+\sigma(Ax_0-b))=A^*y_0=p_0.$$
Then the primal sequence $\seq{u_k}$ is alternatively defined by the following recursion: for every $k\in\mathbb{N}$,
\begin{equation}\label{compare2}
    \begin{split}
        u_{k+1}&=\prox_{\tau R} (u_k-\tau\nabla F(u_k)-\tau q_{k}-\sigma\tau S^2(u_k-x^b))\\
        q_{k+1}&=q_{k}+\sigma S^2(u_{k+1}-x^b) \enspace.
    \end{split}
\end{equation}
Comparing \Cref{compare} with \Cref{compare2}, with $(u_0,q_0)=(x_0,p_0)$, we get the claim.\\
\end{proof}

\subsection{Proof of \Cref{th:unfeas_norm}}
\label{app:unfeas_norm}
\unfeasstocazzo*
\begin{proof}
From \Cref{lem:unfeas_sameit}, we know that the sequence $\seq{\hat{x}_k}$ generated by \Cref{algo:easy} coincides with the primal iterate of a sequence $\seq{\hat{u}_k, \hat{v}_k}$ generated by the same algorithm on problem \eqref{leastsquares}.
Notice that $\nor{S} = \Vert (A^*A)^\frac{1}{2} \Vert = \nor{A}$ and so,
if \Cref{ass:omega} holds, the analogue also holds for problem \eqref{leastsquares}: namely, $1 - \tau(L + \sigma \nor{S}^2) \geq 0$.
The same is true for \Cref{ass:exist}. Indeed, defining $\bar{v}=S\tilde{v}$, $-S\bar{v}=-A^*A\tilde{v}\in\partial R(\tilde{x})+\nabla F(\tilde{x})$.
Moreover, we have seen already that $A^*Ax=A^*b$ if and only if $Sx=Sx^b$, where $x^b$ is any vector in $\X$ such that $A^*Ax^b=A^*b$.
Then, $S\tilde{x}=Sx^b$ and $(\tilde{x},\bar{v})$ is a saddle-point for \eqref{leastsquares}. So, by \Cref{prop:convergence_cp}, we know that the averaged primal-dual sequence $(\hat{u}_k,\hat{v}_k)$ weakly converges to a saddle-point for \eqref{leastsquares}. In particular, there exists $\tilde{x}_{\infty}\in \tilde{\PP}$ such that $\hat{u}_k\rightharpoonup \tilde{x}_{\infty}$ and so the same holds for $(\hat{x}_k)$. For the second claim, by assumption we have that $\PP=\emptyset$, which implies that $\mathcal{S}=\emptyset$.
All the assumptions of \Cref{lem:unfeas} are verified, so $(\hat{x}_k, \hat{y}_k)$ diverges.
As $\seq{\hat{x}_k}$ is weakly convergent and so bounded, we conclude that $\seq{\hat{y}_k}$ has to diverge.
\end{proof}
\subsection{Proof of \Cref{th:unfeas_stab}}\label{app:unfeas_3}
\unfeasstab*

\begin{proof}
From the assumption $\tilde{\mathcal{C}}^{\delta}\neq \emptyset$ and \Cref{lem:unfeas_sameit}, we know that the sequence $\seq{\hat{x}_k}$ coincides with the primal iterate of a sequence $\seq{\hat{u}_k, \hat{v}_k}$ generated by the same algorithm on problem \begin{equation}\label{Ptildedelta}
    \tilde{\PP}^{\delta}=\argmin_{x\in\X} \left\{ R(x)+F(x): \ \ Sx=Sx^{\delta} \right\},
\end{equation}
where $x^{\delta}$ is any vector in $\X$ such that $A^*Ax^{\delta}=A^*b^{\delta}$. As in the proof of the previous theorem, notice that $\nor{S} = \nor{A}$ and so,
as \Cref{ass:omega} and \Cref{ass:thetarho} hold by hypothesis, the analogue also holds for problem \eqref{Ptildedelta}: namely, $1 - \tau(L + \sigma \nor{S}^2) \geq 0$, $\xi - \tau(\xi L+ \sigma \nor{S}^2)\geq 0$ and $\sigma(\eta - 1) - \sigma \xi \eta>0$. The same is true for \Cref{ass:exist}. Indeed, define $\bar{v}=S\tilde{v}$. Then, from \Cref{optcond}, $-S\bar{v}=-A^*A\tilde{v}\in\partial R(\tilde{x})+\nabla F(\tilde{x})$ and $(\tilde{x},\bar{v})$ is a saddle-point for
\begin{equation}\label{eq2}
\tilde{\PP}^{\star}=\argmin_{x\in\X} \left\{ R(x)+F(x): \ \ Sx=S\tilde{x} \right\}.
\end{equation}
In particular, we can apply \Cref{prop:early_stop} for $\seq{\hat{u}_k, \hat{v}_k}$ - averaged primal-dual sequence generated on the noisy problem in \eqref{Ptildedelta} - with respect to $(\tilde{x},\bar{v})$ - saddle-point for the exact problem in \eqref{eq2} - to get that
\begin{equation*}
    D^{-S\bar{v}}(\hat{u}_k,\tilde{x}) \leq \frac{C_1}{k} +C_2 \tilde{\delta} + C_4 (\tilde{\delta})^2  k
\end{equation*}
and
\begin{equation*}
    \norm{S\hat{u}_k-S\tilde{x}}{}^2\leq \frac{C_5}{k} + C_6 \tilde{\delta}+  C_8 (\tilde{\delta})^2 k  + C_9 (\tilde{\delta})^2.
\end{equation*}
The constants in the previous bounds are the same as in \eqref{constants} with $z_0=(u_0,v_0)$, $z^{\star}=(\tilde{x},\bar{v})$, $C_3=C_7=0$ (because $C_0=0$ as we suppose $\varepsilon_k=0$ for every $k\in\N$), $\sigma_m=\sigma_M=\sigma$ and
\begin{equation*}
    \tilde{\delta}:=\|Sx^{\delta}-S\tilde{x}\|.
\end{equation*}
From \Cref{lem:unfeas_sameit}, we recall also that $u_0=x_0$ and $v_0=v_{-1}+\sigma(Su_0-Sx^{\delta})$, where $v_{-1}$ is any element in $\X$ such that $Sv_{-1}=A^*y_{-1}$ ($v_{-1}$ exists due to $R(A^*)=R(S)$). Now it remains to show that $\tilde{\delta}\leq \delta$. Denote by $(\mu_i, f_i,g_i )_{i\in\N} \subseteq \R_+ \times \X \times \Y$ the singular value decomposition of the operator $A$. First, notice that $S^2(x^{\delta}-\tilde{x})=A^*(b^{\delta}-b^{\star})$ and so that, for every $i\in\N$,
$$\mu_i^2 \langle x^{\delta}-\tilde{x},f_i\rangle=\mu_i \langle b^{\delta}-b^{\star}, g_i \rangle.$$
Then, for every $i\in\N$ such that $\mu_i\neq 0$, $\mu_i \langle x^{\delta}-\tilde{x},f_i\rangle=\langle b^{\delta}-b^{\star}, g_i \rangle$ and so
\begin{align*}
\tilde{\delta}^2 & = \|Sx^{\delta}-S\tilde{x}\|^2 =\sum_{i\in\N} \left( \mu_i \langle x^{\delta}-\tilde{x},f_i\rangle \right) ^2  =\sum_{\mu_i\neq 0} \left( \mu_i \langle x^{\delta}-\tilde{x},f_i\rangle \right) ^2 \\
& = \sum_{\mu_i\neq 0} \left( \langle b^{\delta}-b^{\star}, g_i \rangle  \right)^2 \leq \sum_{i\in\N} \left( \langle b^{\delta}-b^{\star}, g_i \rangle  \right)^2 = \|b^{\delta}-b^{\star}\|^2 \leq \delta^2.
\end{align*}
We conclude the claim simply by noticing that $$D^{-A^*A\tilde{v}}(\hat{x}_k,\tilde{x})=D^{-S\bar{v}}(\hat{u}_k,\tilde{x})$$
and
$$\norm{A^*A\hat{x}_k-A^*b^{\star}}{}=\norm{S^2\hat{u}_k-S^2 \tilde{x}}{}\leq \norm{S}{}\norm{S\hat{u}_k-S\tilde{x}}{}.$$
\end{proof}

\section{A dual view on the implicit bias of gradient descent on least squares}
\label{app:implicit_dual}

Here we provide an interesting view on why the ``implicit'' bias of gradient descent on least squares is not so implicit.
Recall that these iterations,
\begin{equation}\label{eq:gd_ls}
    x_{k+1} = x_k - \gamma A^* (A x_k - b) \enspace,
\end{equation}
converge, for $\gamma < 2 / \nor{A}^2_{\mathrm{op}}$, to the minimal Euclidean norm solution of $Ax = b$:
\begin{equation}\label{eq:min_norm}
    \min_{x \in \X} \frac12 \nor{x}^2     \quad \text{s.t.} \quad Ax = b \enspace,
\end{equation}
provided that Problem \eqref{eq:min_norm} is feasible and $x_0=0$.

It turns out that the iterations \eqref{eq:gd_ls} correspond, up to multiplication by $-A^*$, to the iterates of gradient descent to the dual of \eqref{eq:min_norm}, namely:
\begin{equation}\label{eq:dual_min_norm}
    \min_{y \in \Y} \frac12 \nor{A^* y}^2 + \langle b, y \rangle \enspace, \quad \text{and} \quad y_{k+1} = y_k - \gamma (A A^* y_k + b) \enspace.
\end{equation}
By setting $x_{k+1} =- A^* y_{k+1}$ one recovers the iterates of gradient descent on least squares \eqref{eq:gd_ls}.
Therefore the ``implicit bias'' of gradient descent on least squares is not so implicit:
its iterates $x_k$ are dual to iterates $y_k$ on Problem \eqref{eq:dual_min_norm}, which is itself the dual of Problem \eqref{eq:min_norm} in which the bias appears explicitly.

\bibliographystyle{spmpsci}      %

\begin{thebibliography}{10}
	\providecommand{\url}[1]{{#1}}
	\providecommand{\urlprefix}{URL }
	\expandafter\ifx\csname urlstyle\endcsname\relax
	  \providecommand{\doi}[1]{DOI~\discretionary{}{}{}#1}\else
	  \providecommand{\doi}{DOI~\discretionary{}{}{}\begingroup
	  \urlstyle{rm}\Url}\fi

	\bibitem{Akaike74}
	Akaike, H.: A new look at the statistical model identification.
	\newblock IEEE Trans. Automat. Control \textbf{AC-19}, 716--723 (1974)

	\bibitem{bach2011optimization}
	Bach, F., Jenatton, R., Mairal, J., Obozinski, G.: Optimization with
	  sparsity-inducing penalties.
	\newblock arXiv preprint arXiv:1108.0775  (2011)

	\bibitem{bach2012structured}
	Bach, F., Jenatton, R., Mairal, J., Obozinski, G.: Structured sparsity through
	  convex optimization.
	\newblock Statistical Science \textbf{27}(4), 450--468 (2012)

	\bibitem{bachmayr2009iterative}
	Bachmayr, M., Burger, M.: Iterative total variation schemes for nonlinear
	  inverse problems.
	\newblock Inverse Problems \textbf{25}(10), 105004 (2009)

	\bibitem{BahLem94}
	Bahraoui, M., Lemaire, B.: Convergence of diagonally stationary sequences in
	  convex optimization.
	\newblock Set-Valued Anal. \textbf{2}, 49--61 (1994)

	\bibitem{barre2020principled}
	Barr{\'e}, M., Taylor, A., Bach, F.: Principled analyses and design of
	  first-order methods with inexact proximal operators.
	\newblock arXiv preprint arXiv:2006.06041  (2020)

	\bibitem{Bauschke_Combettes11}
	Bauschke, H.H., Combettes, P.L.: Convex analysis and monotone operator theory
	  in {Hilbert} spaces.
	\newblock Springer, New York (2011)

	\bibitem{Boyd_Parikh_Chu_Peleato_Eckstein11}
	Boyd, S., Parikh, N., Chu, E., Peleato, B., Eckstein, J.: Distributed
	  optimization and statistical learning via the alternating direction method of
	  multipliers.
	\newblock Foundations and Trends in Machine Learning \textbf{3}(1), 1--122
	  (2011)

	\bibitem{brianzi2013preconditioned}
	Brianzi, P., Di~Benedetto, F., Estatico, C.: Preconditioned iterative
	  regularization in banach spaces.
	\newblock Computational Optimization and Applications \textbf{54}(2), 263--282
	  (2013)

	\bibitem{burger2007error}
	Burger, M., Resmerita, E., He, L.: Error estimation for {Bregman} iterations
	  and inverse scale space methods in image restoration.
	\newblock Computing \textbf{81}(2-3), 109--135 (2007)

	\bibitem{cai2009convergence}
	Cai, J.F., Osher, S., Shen, Z.: Convergence of the linearized {Bregman}
	  iteration for $\ell_1$-norm minimization.
	\newblock Mathematics of Computation \textbf{78}(268), 2127--2136 (2009)

	\bibitem{calatroni2019accelerated}
	Calatroni, L., Garrigos, G., Rosasco, L., Villa, S.: Accelerated iterative
	  regularization via dual diagonal descent.
	\newblock arXiv preprint arXiv:1912.12153  (2019)

	\bibitem{Candes_Recht09}
	Cand{\`e}s, E.J., Recht, B.: Exact matrix completion via convex optimization.
	\newblock Found. Comput. Math. \textbf{9}(6), 717--772 (2009)

	\bibitem{Candes_Romberg_Tao06}
	Cand{\`e}s, E.J., Romberg, J., Tao, T.: Robust uncertainty principles: exact
	  signal reconstruction from highly incomplete frequency information.
	\newblock {IEEE} Trans. Inf. Theory \textbf{52}(2), 489--509 (2006)

	\bibitem{chambolle2011first}
	Chambolle, A., Pock, T.: A first-order primal-dual algorithm for convex
	  problems with applications to imaging.
	\newblock Journal of mathematical imaging and vision \textbf{40}(1), 120--145
	  (2011)

	\bibitem{Chambolle_Pock11}
	Chambolle, A., Pock, T.: A first-order primal-dual algorithm for convex
	  problems with applications to imaging.
	\newblock J. Math. Imaging Vis. \textbf{40}(1), 120--145 (2011)

	\bibitem{Chen_Donoho_Saunders98}
	Chen, S.S., Donoho, D.L., Saunders, M.A.: Atomic decomposition by basis
	  pursuit.
	\newblock SIAM J. Sci. Comput. \textbf{20}(1), 33--61 (1998)

	\bibitem{chizat2020implicit}
	Chizat, L., Bach, F.: Implicit bias of gradient descent for wide two-layer
	  neural networks trained with the logistic loss.
	\newblock In: Conference on Learning Theory, pp. 1305--1338 (2020)

	\bibitem{combettes2011proximal}
	Combettes, P.L., Pesquet, J.C.: Proximal splitting methods in signal
	  processing.
	\newblock In: Fixed-point algorithms for inverse problems in science and
	  engineering, pp. 185--212. Springer (2011)

	\bibitem{Condat13}
	Condat, L.: A primal--dual splitting method for convex optimization involving
	  {Lipschitzian}, proximable and linear composite terms.
	\newblock Journal of Optimization Theory and Applications \textbf{158}(2),
	  460--479 (2013)

	\bibitem{Devroye_Wagner1979}
	Devroye, L., Wagner, T.: Distribution-free performance bounds for potential
	  function rules.
	\newblock IEEE Transactions on Information Theory \textbf{25}(5), 601--604
	  (1979)

	\bibitem{Donoho06}
	Donoho, D.L.: Compressed sensing.
	\newblock {IEEE} Trans. Inf. Theory \textbf{52}(4), 1289--1306 (2006)

	\bibitem{elvira2020safe}
	Elvira, C., Herzet, C.: Safe squeezing for antisparse coding.
	\newblock IEEE Transactions on Signal Processing \textbf{68}, 3252--3265 (2020)

	\bibitem{engl1996regularization}
	Engl, H.W., Heinz, W., Hanke, M., Neubauer, A.: Regularization of inverse
	  problems, vol. 375.
	\newblock Springer Science \& Business Media (1996)

	\bibitem{Fazel02}
	Fazel, M.: Matrix rank minimization with applications.
	\newblock Ph.D. thesis, Stanford University (2002)

	\bibitem{figueiredo2016ordered}
	Figueiredo, M., Nowak, R.: Ordered weighted $\ell_1$ regularized regression
	  with strongly correlated covariates: Theoretical aspects.
	\newblock In: AISTATS, pp. 930--938. PMLR (2016)

	\bibitem{foucart2013}
	Foucart, S., Rauhut, H.: A Mathematical Introduction to Compressive Sensing.
	\newblock Springer, New York (2013)

	\bibitem{friedlander2008exact}
	Friedlander, M.P., Tseng, P.: Exact regularization of convex programs.
	\newblock SIAM Journal on Optimization \textbf{18}(4), 1326--1350 (2008)

	\bibitem{garrigos2018iterative}
	Garrigos, G., Rosasco, L., Villa, S.: Iterative regularization via dual
	  diagonal descent.
	\newblock Journal of Mathematical Imaging and Vision \textbf{60}(2), 189--215
	  (2018)

	\bibitem{montanari}
	Ghorbani, B., Mei, S., Misiakiewicz, T., Montanari, A.: Linearized two-layers
	  neural networks in high dimension.
	\newblock The Annals of Statistics \textbf{49}(2), 1029--1054 (2021)

	\bibitem{grasmair2011necessary}
	Grasmair, M., Scherzer, O., Haltmeier, M.: Necessary and sufficient conditions
	  for linear convergence of l1-regularization.
	\newblock Communications on Pure and Applied Mathematics \textbf{64}(2),
	  161--182 (2011)

	\bibitem{gunasekar2018characterizing}
	Gunasekar, S., Lee, J., Soudry, D., Srebro, N.: Characterizing implicit bias in
	  terms of optimization geometry.
	\newblock arXiv preprint arXiv:1802.08246  (2018)

	\bibitem{gunasekar2017implicit}
	Gunasekar, S., Woodworth, B.E., Bhojanapalli, S., Neyshabur, B., Srebro, N.:
	  Implicit regularization in matrix factorization.
	\newblock In: NeurIPS, pp. 6151--6159 (2017)

	\bibitem{Hastie_Tibshirani_Friedman09}
	Hastie, T.J., Tibshirani, R., Friedman, J.: The {Elements} of {Statistical}
	  {Learning}, second edn.
	\newblock Springer Series in Statistics. Springer, New York (2009)

	\bibitem{Hastie_Tibshirani_Wainwright15}
	Hastie, T.J., Tibshirani, R., Wainwright, M.: Statistical Learning with
	  Sparsity: The Lasso and Generalizations.
	\newblock CRC Press (2015)

	\bibitem{iutzeler2020nonsmoothness}
	Iutzeler, F., Malick, J.: Nonsmoothness in machine learning: specific
	  structure, proximal identification, and applications.
	\newblock Set-Valued and Variational Analysis \textbf{28}(4), 661--678 (2020)

	\bibitem{kaltenbacher2008iterative}
	Kaltenbacher, B., Neubauer, A., Scherzer, O.: Iterative regularization methods
	  for nonlinear ill-posed problems, vol.~6.
	\newblock Walter de Gruyter (2008)

	\bibitem{lorenz2014linearized}
	Lorenz, D.A., Schopfer, F., Wenger, S.: The linearized bregman method via split
	  feasibility problems: Analysis and generalizations.
	\newblock SIAM Journal on Imaging Sciences \textbf{7}(2), 1237--1262 (2014)

	\bibitem{massias2019dual}
	Massias, M., Vaiter, S., Gramfort, A., Salmon, J.: Dual extrapolation for
	  sparse generalized linear models.
	\newblock JMLR  (2020)

	\bibitem{matet2017don}
	Matet, S., Rosasco, L., Villa, S., Vu, B.L.: Don't relax: early stopping for
	  convex regularization.
	\newblock arXiv preprint arXiv:1707.05422  (2017)

	\bibitem{mosci2010solving}
	Mosci, S., Rosasco, L., Santoro, M., Verri, A., Villa, S.: Solving structured
	  sparsity regularization with proximal methods.
	\newblock In: Joint European conference on machine learning and knowledge
	  discovery in databases, pp. 418--433. Springer (2010)

	\bibitem{moulines2011non}
	Moulines, E., Bach, F.: Non-asymptotic analysis of stochastic approximation
	  algorithms for machine learning.
	\newblock In: NeurIPS, pp. 451--459 (2011)

	\bibitem{Nemirovski_Yudin83}
	Nemirovski, A.S., Yudin, D.B.: Problem complexity and method efficiency in
	  optimization.
	\newblock A Wiley-Interscience Publication. John Wiley \& Sons Inc., New York
	  (1983)

	\bibitem{nesterov1983method}
	Nesterov, Y.E.: A method for solving the convex programming problem with
	  convergence rate o (1/k\^{} 2).
	\newblock In: Dokl. akad. nauk Sssr, vol. 269, pp. 543--547 (1983)

	\bibitem{neubauer2017nesterov}
	Neubauer, A.: On nesterov acceleration for landweber iteration of linear
	  ill-posed problems.
	\newblock Journal of Inverse and Ill-posed Problems \textbf{25}(3), 381--390
	  (2017)

	\bibitem{Obozinski_Taskar_Jordan10}
	Obozinski, G., Taskar, B., Jordan, M.I.: Joint covariate selection and joint
	  subspace selection for multiple classification problems.
	\newblock Statistics and Computing \textbf{20}(2), 231--252 (2010)

	\bibitem{OshBurGol05}
	Osher, S., Burger, M., Goldfarb, D., Xu, J., Yin, W.: An iterative
	  regularization method for total variation-based image restoration.
	\newblock SIAM Multiscale Model. Simul. \textbf{4}, 460--489 (2005)

	\bibitem{osher2016sparse}
	Osher, S., Ruan, F., Xiong, J., Yao, Y., Yin, W.: Sparse recovery via
	  differential inclusions.
	\newblock Applied and Computational Harmonic Analysis \textbf{41}(2), 436--469
	  (2016)

	\bibitem{paglia}
	Pagliana, N., Rosasco, L.: Implicit regularization of accelerated methods in
	  {H}ilbert spaces.
	\newblock In: NeurIPRS, pp. 14454--14464 (2019)

	\bibitem{pock2011diagonal}
	Pock, T., Chambolle, A.: Diagonal preconditioning for first order primal-dual
	  algorithms in convex optimization.
	\newblock In: 2011 International Conference on Computer Vision, pp. 1762--1769
	  (2011)

	\bibitem{polyak1964some}
	Polyak, B.T.: Some methods of speeding up the convergence of iteration methods.
	\newblock Ussr computational mathematics and mathematical physics
	  \textbf{4}(5), 1--17 (1964)

	\bibitem{raskutti}
	Raskutti, G., Wainwright, M.J., Yu, B.: Early stopping and non-parametric
	  regression: An optimal data-dependent stopping rule.
	\newblock J. Mach. Learn. Res. \textbf{15}(1), 335–366 (2014)

	\bibitem{rosasco2015learning}
	Rosasco, L., Villa, S.: Learning with incremental iterative regularization.
	\newblock In: NeurIPS, pp. 1630--1638 (2015)

	\bibitem{Rudin_Osher_Fatemi92}
	Rudin, L.I., Osher, S., Fatemi, E.: Nonlinear total variation based noise
	  removal algorithms.
	\newblock Phys. D \textbf{60}(1-4), 259--268 (1992)

	\bibitem{salzo2012inexact}
	Salzo, S., Villa, S.: Inexact and accelerated proximal point algorithms.
	\newblock Journal of Convex analysis \textbf{19}(4), 1167--1192 (2012)

	\bibitem{schmidt2011convergence}
	Schmidt, M., Le~Roux, N., Bach, F.: Convergence rates of inexact
	  proximal-gradient methods for convex optimization.
	\newblock In: NeurIPS, pp. 1458--1466 (2011)

	\bibitem{schopfer2012exact}
	Schopfer, F.: Exact regularization of polyhedral norms.
	\newblock SIAM Journal on Optimization \textbf{22}(4), 1206--1223 (2012)

	\bibitem{schopfer2019linear}
	Sch{\"o}pfer, F., Lorenz, D.A.: Linear convergence of the randomized sparse
	  kaczmarz method.
	\newblock Mathematical Programming \textbf{173}(1), 509--536 (2019)

	\bibitem{schopfer2006nonlinear}
	Sch{\"o}pfer, F., Louis, A.K., Schuster, T.: Nonlinear iterative methods for
	  linear ill-posed problems in banach spaces.
	\newblock Inverse problems \textbf{22}(1), 311 (2006)

	\bibitem{Schwarz78}
	Schwarz, G.: Estimating the dimension of a model.
	\newblock AOS \textbf{6}(2), 461--464 (1978)

	\bibitem{shai}
	Shalev-Shwartz, S., Ben-David, S.: Understanding Machine Learning: From Theory
	  to Algorithms.
	\newblock Cambridge eBooks (2014)

	\bibitem{Simon_Friedman_Hastie_Tibshirani12}
	Simon, N., Friedman, J., Hastie, T.J., Tibshirani, R.: A sparse-group lasso.
	\newblock J. Comput. Graph. Statist. \textbf{22}(2), 231--245 (2013)

	\bibitem{BecTeb03}
	Teboulle, M., Beck, A.: Mirror descent and nonlinear projected subgradient
	  methods for convex optimization.
	\newblock Oper. Res. Letters \textbf{31}, 167--175 (2003)

	\bibitem{Tibshirani96}
	Tibshirani, R.: Regression shrinkage and selection via the lasso.
	\newblock J. R. Stat. Soc. Ser. B Stat. Methodol. \textbf{58}(1), 267--288
	  (1996)

	\bibitem{vaiter2015low}
	Vaiter, S., Peyr{\'e}, G., Fadili, J.: Low complexity regularization of linear
	  inverse problems.
	\newblock In: Sampling Theory, a Renaissance, pp. 103--153. Springer (2015)

	\bibitem{Vaskevicius_Kanade_Rebeschini19}
	Va{\v{s}}kevi{\v{c}}ius, T., Kanade, V., Rebeschini, P.: Implicit
	  regularization for optimal sparse recovery.
	\newblock In: NeurIPS, pp. 2968--2979 (2019)

	\bibitem{vavskevivcius2020statistical}
	Va{\v{s}}kevi{\v{c}}ius, T., Kanade, V., Rebeschini, P.: The statistical
	  complexity of early stopped mirror descent.
	\newblock arXiv preprint arXiv:2002.00189  (2020)

	\bibitem{villa2013accelerated}
	Villa, S., Salzo, S., Baldassarre, L., Verri, A.: Accelerated and inexact
	  forward-backward algorithms.
	\newblock SIAM Journal on Optimization \textbf{23}(3), 1607--1633 (2013)

	\bibitem{Vu13}
	V{\~u}, B.C.: A splitting algorithm for dual monotone inclusions involving
	  cocoercive operators.
	\newblock Advances in Computational Mathematics \textbf{38}(3), 667--681 (2013)

	\bibitem{yao2007early}
	Yao, Y., Rosasco, L., Caponnetto, A.: On early stopping in gradient descent
	  learning.
	\newblock Constructive Approximation \textbf{26}(2), 289--315 (2007)

	\bibitem{yin2010analysis}
	Yin, W.: Analysis and generalizations of the linearized {Bregman} method.
	\newblock SIAM Journal on Imaging Sciences \textbf{3}(4), 856--877 (2010)

	\bibitem{YinOshBur08}
	Yin, W., Osher, S., Goldfarb, D., Darbon, J.: Bregman iterative algorithms for
	  l1- minimization with applications to compressed sensing.
	\newblock SIAM J. Imaging Sci. \textbf{1}(1), 143--168 (2008)

	\bibitem{ZhaBurBre10}
	Zhang, X., Burger, M., Bresson, X., Osher, S.: Bregmanized nonlocal
	  regularization for deconvolution and sparse reconstruction.
	\newblock SIAM J. Imaging Sci. \textbf{3}, 253--276 (2010)

	\bibitem{ZhaBurOsh11}
	Zhang, X., Burger, M., Osher, S.: A unified primal-dual algorithm framework
	  based on {B}regman iteration.
	\newblock J. Sci. Comput. \textbf{46}, 20--46 (2011)

\end{thebibliography}

	 %

\end{document}